\documentclass[12pt]{amsart}
\usepackage[paperheight=11in,paperwidth=8.5in,left=1in,top=1.25in,right=1in,bottom=1.25in,]{geometry}
\usepackage{amssymb}

\usepackage{pict2e, epsfig, amstext}
\usepackage{youngtab}
\usepackage[boxsize=.35 em]{ytableau}
\usepackage{varioref}
\usepackage{hyperref}
\usepackage[capitalize,nameinlink,noabbrev,nosort]{cleveref}
\usepackage{tkz-euclide}
\usepackage{manfnt}
\usepackage[framemethod=tikz,ntheorem,xcolor]{mdframed}
\usepackage{parcolumns}
\usepackage{tabularx}

\usepackage{tikz}\pdfpageattr{/Group <</S /Transparency /I true /CS /DeviceRGB>>}
\usetikzlibrary{arrows,calc,intersections,matrix,positioning,through}
\usepackage{tikz-cd}
\tikzset{commutative diagrams/.cd,every label/.append style = {font = \normalsize}}

\usepgflibrary{shapes.geometric}
\usetikzlibrary{shapes.geometric}

\numberwithin{equation}{section}
\newtheorem*{theorem*}{Theorem}
\newtheorem*{corollary*}{Corollary}
\newtheorem{thm}[equation]{Theorem}
\AfterEndEnvironment{thm}{\noindent\ignorespaces}
\newtheorem{theorem}[equation]{Theorem}
\AfterEndEnvironment{theorem}{\noindent\ignorespaces}

\newtheorem{corollary}[equation]{Corollary}
\AfterEndEnvironment{cor}{\noindent\ignorespaces}
\newtheorem{lemma}[equation]{Lemma}
\AfterEndEnvironment{lemma}{\noindent\ignorespaces}

\AfterEndEnvironment{lem}{\noindent\ignorespaces}
\newtheorem{proposition}[equation]{Proposition}
\AfterEndEnvironment{proposition}{\noindent\ignorespaces}
\newtheorem{prop}[equation]{Proposition}
\AfterEndEnvironment{prop}{\noindent\ignorespaces}
\newtheorem{conj}[equation]{Conjecture}
\AfterEndEnvironment{conj}{\noindent\ignorespaces}

\AfterEndEnvironment{prob}{\noindent\ignorespaces}

\AfterEndEnvironment{question}{\noindent\ignorespaces}
\AfterEndEnvironment{remark}{\noindent\ignorespaces}

\theoremstyle{definition}
\newtheorem{defn}[equation]{Definition}
\newtheorem{definition}[equation]{Definition}
\newtheorem{example}[equation]{Example}
\newtheorem{remark}[equation]{Remark}

\AfterEndEnvironment{defn}{\noindent\ignorespaces}
\AfterEndEnvironment{definition}{\noindent\ignorespaces}

\newcommand{\td}[1]{\hat{#1}} 

\def\MR#1{}

\def\M{\mathcal{M}}
\def\B{\mathcal{B}}

\newcommand{\hatG}{\hat{G}}
\def\mcv{\mathcal{V}}

\newcommand{\Grk}{\Gr_{k,n}^{\ge 0}}
\newcommand{\A}{\mathcal{A}}
\newcommand{\T}{\mathcal{T}}

\def\AA{\mathcal{A}_{n,k,m}}
\newcommand{\R}{\mathbb{R}}
\newcommand{\RR}{\mathbb{R}}
\newcommand{\CC}{\mathbb{C}}
\newcommand{\CCC}{\mathcal{C}}

\def\PP{\mathbb{P}}
\newcommand{\gt}[1]{Z_{#1}} 
\newcommand{\gto}[1]{Z_{#1}^\circ} 
\newcommand{\simp}[1]{\Delta_{#1}}
\newcommand{\asimp}[1]{\hat{\Delta}_{#1}(Z)}
\newcommand{\asimpo}[1]{\hat{\Delta}_{#1}^{\circ}(Z)}

\DeclareMathOperator{\Mat}{Mat}
\DeclareMathOperator{\rank}{rank}
\DeclareMathOperator{\Trop}{Trop}

\DeclareMathOperator{\convex}{conv}
\DeclareMathOperator{\spn}{span}

\DeclareMathOperator{\Facet}{Facet}
\DeclareMathOperator{\area}{area}

\DeclareMathOperator{\Edges}{Edges}

\DeclareMathOperator{\var}{var}
\DeclareMathOperator{\sign}{sign}

\DeclareMathOperator{\cdes}{cDes_L}
\DeclareMathOperator{\Gr}{Gr}
\DeclareMathOperator{\Dr}{Dr}

\newcommand{\lrangle}[1]{\langle #1 \rangle}

\begin{document}


\title[The positive Grassmannian, amplituhedra, and clusters]{The positive Grassmannian, the amplituhedron, \\ and cluster algebras}

\author{Lauren K. Williams}


\address{Department of Mathematics, 1 Oxford Street, Cambridge, MA 02138}
\email{{williams@math.harvard.edu}}



\begin{abstract}

The \emph{positive Grassmannian} $\Gr_{k,n}^{\geq 0}$
is the subset of the real Grassmannian where all 
Pl\"ucker coordinates are nonnegative.  
It has a beautiful combinatorial structure
as well as connections to statistical physics,
	integrable systems, and scattering amplitudes.
The \emph{amplituhedron}
$\mathcal{A}_{n,k,m}(Z)$
is the image of the positive Grassmannian $\Gr_{k,n}^{\geq 0}$ under
a positive linear map $\R^n \to \R^{k+m}$.  We will explain 
how ideas from 
oriented matroids, tropical geometry,
 and cluster algebras shed light on the structure of the positive
	Grassmannian and the amplituhedron.
\end{abstract}

\maketitle
\setcounter{tocdepth}{1}
\tableofcontents


\section{Introduction}

The \emph{totally nonnegative 
Grassmannian} $\Grk$ or (informally) the \emph{positive Grassmannian} \cite{lusztig, postnikov}
can be defined as the subset of the real Grassmannian $\Gr_{k,n}$ where all 
Pl\"ucker coordinates
are nonnegative.  
It has a beautiful decomposition into \emph{positroid cells} \cite{postnikov, rietsch, PSW}, 
where each cell is obtained by specifying that certain Pl\"ucker coordinates are strictly
positive and the rest are zero.  
Since the work of Lusztig \cite{lusztig} and Postnikov \cite{postnikov, PostnikovICM}, 
there has been an extensive study of 
the positive Grassmannian,
including approaches involving cluster algebras,
tropical geometry, and 
  matroids and the moment map.

Remarkably, the positive Grassmannian has
several  applications in theoretical physics.
For example, the stationary distribution of the 
\emph{asymmetric simple exclusion process}, 
in which  particles hop on a line
with open boundaries, can be described in terms of cells of $\Gr_{k,n}^{\geq 0}$
\cite{CW1}.  In a different direction, 
each point $C$ of the real Grassmannian gives rise to a
\emph{soliton solution of the KP equation}, whose asymptotics are determined by 
the matroid of $C$, and which is regular for all times $t$ 
if and only if $C$ lies in the positive Grassmannian \cite{KW, KW2}.
In yet a third direction, 
 the {positive} Grassmannian encodes most of the physical properties of 
 \emph{scattering
amplitudes in planar $\mathcal{N}=4$ super Yang-Mills theory}
 \cite{ArkaniHamed:2009dn,  Bullimore:2009cb, abcgpt}. 
 This insight 
 combined with 
 an idea of  Hodges
 \cite{Hodges:2009hk} led Arkani-Hamed and Trnka to introduce the 
 \emph{amplituhedron} \cite{arkani-hamed_trnka},
 defined as the image
 of the positive Grassmannian under the \emph{amplituhedron map}.
In particular, any $n \times (k+m)$
matrix $Z$ whose maximal minors are positive induces
a map $\tilde{Z}$ from $\Gr^{\geq 0}_{k,n}$ to the Grassmannian
$\Gr_{k,k+m}$, whose image (of full dimension $km$)  is 
the \emph{amplituhedron} $\A_{n,k,m}(Z)$ \cite{arkani-hamed_trnka}.
When $m=4$,
 the \emph{BCFW recurrence} \cite{BCFW}
for computing scattering amplitudes can be used to produce
collections of $4k$-dimensional cells in $\Gr^{\geq 0}_{k,n}$ whose images 
subdivide (``tile'' or ``triangulate'') the amplituhedron \cite{arkani-hamed_trnka, ELT}.


The amplituhedron $\AA(Z)$ generalizes the positive Grassmannian (obtained 
when $k+m=n$),
cyclic polytopes (when $k=1$) \cite{arkani-hamed_trnka},  and cyclic hyperplane arrangements (when $m=1$)
\cite{karpwilliams}.
Moreover the amplituhedron has 
intriguing and beautiful
mathematical properties, many of them conjectural. 
For instance,
we conjecture that for even $m$, the number of top-dimensional strata
comprising  a tiling of $\AA(Z)$ is equal to
the number of plane partitions contained in the $k \times (n-k-m) \times \frac{m}{2}$
box \cite{karp:2017ouj}.
As another example, despite the fact that they have different dimensions
and one of them is not a polytope,
the hypersimplex $\Delta_{k+1,n}$ and the amplituhedron $\A_{n,k,2}(Z)$
are closely related: 
for example, \emph{T-duality} gives a bijection between
positroid tilings of $\Delta_{k+1,n}$
and positroid tilings of $\A_{n,k,2}(Z)$ \cite{LPW, PSBW}.

In this article 
we explain how ideas from 
the theory of matroids, tropical geometry,
 and cluster algebras shed light on the structure of 
 positive Grassmannians and amplituhedra.
We start in \cref{sec:strat} by introducing
the matroid stratification of the Grassmannian
and the positroid cell decomposition of the positive Grassmannian.
Given a surjective map $\phi$ from a cell complex $X$ onto another 
topological space $Y$, we also introduce the notion of \emph{$\phi$-induced 
tiling} of $Y$, which we will study in the case that $X$
is the positive Grassmannian (and call a \emph{positroid tiling}).
In \cref{sec:moment} 
we study positroid tilings when $\phi$ is the moment map,
which are subdivisions of the 
hypersimplex into positroid polytopes, and are related to the 
positive tropical Grassmannian.
In \cref{sec:stratification} we introduce the amplituhedron, 
giving two equivalent definitions, 
defining natural coordinates, characterizing its points when $m=1$ and $2$,
and defining
its \emph{sign stratification}, which is an analogue of the matroid 
stratification. 
In \cref{sec:amp} we then study positroid tilings when $\phi$
is the  amplituhedron
map.
We give a conjectural link to plane partitions, and discuss
the positroid cells on which the amplituhedron map is injective.
In \cref{sec:T} we explain a mysterious notion called \emph{T-duality}, which 
relates positroid tiles and tilings of the hypersimplex $\Delta_{k+1,n}$ to 
positroid tiles and tilings for  the amplituhedron $\mathcal{A}_{n,k,2}(Z)$. One manifestation
of this duality is the fact that the number of realizable sign strata
of $\A_{n,k,2}(Z)$ equals the volume of $\Delta_{k+1,n}$ (an Eulerian number).
Finally in \cref{sec:cluster} we present several connections
between the amplituhedron and cluster algebras, proved for $m=2$
but conjectural in general.

A great many mathematicians and physicists have
made tremendous contributions to the study of the positive 
Grassmannian and amplituhedron; it is impossible to give 
a complete account here.  The results described below in which
I played a role are joint with 
various 
collaborators
including F. Ardila, S. Karp, 
T. Lukowski, M. Parisi,  
K. Rietsch, F. Rinc\'on, M. Sherman-Bennett, D. Speyer,
K. Talaska, 
E. Tsukerman, and Y. Zhang.

\section{The positive Grassmannian and the matroid stratification}\label{sec:strat}

\subsection{The Grassmannian and the matroid stratification}
The \emph{Grassmannian} $\Gr_{k,n}=\Gr_{k,n}(\mathbb{K})$  
is the space of all $k$-dimensional subspaces of 
an $n$-dimensional vector space $\mathbb{K}^n$.  
Let $[n]$ denote $\{1,\dots,n\}$, and $\binom{[n]}{k}$ denote the set of all $k$-element subsets of $[n]$. 
We can
 represent a point $V \in 
\Gr_{k,n}$  as the row-span 
of
a full-rank $k\times n$ matrix $C$ with entries in 
$\mathbb{K}$;
then, for $I\in \binom{[n]}{k}$, we let $p_I(V)$ be the $k\times k$ minor of $C$ occupying the columns in $I$. The $p_I(V)$ 
are called the {\itshape Pl\"{u}cker coordinates} of $V$, and are independent of the choice of matrix 
representative $C$ (up to common rescaling).  The map 
$V \mapsto \{p_I(V)\}_{I\in \binom{[n]}{k}}$ 
embeds  $\Gr_{k,n}$ into 
projective space.  
We will sometimes abuse notation and \emph{identify $C$ with its row-span.}

\begin{defn}\label{defn:matroid}
A \emph{matroid} $\M$ is a pair $(E,\B)$, where $E$ is a finite set
and $\B$  a nonempty 
collection of subsets of $E$ called \emph{bases}, such 
that if $B_1, B_2$ are distinct bases and $b_1\in B_1 \setminus B_2$,
then there exists an element $b_2\in B_2 \setminus B_1$ such that
$(B_1\setminus \{b_1\}) \cup \{b_2\}$ is a basis.
\end{defn}

\emph{Matroid theory} originated in the 1930's as a combinatorial model
that keeps track of, and abstracts, the dependence relations among
a set of vectors.  

\begin{definition}\label{def:realizable}
Any full-rank 
	$k\times n$ matrix $C$  (with entries in a field
	$\mathbb{K}$), and consequently
	any point $C\in \Gr_{k,n}(\mathbb{K})$, gives rise to a matroid
$\M(C):=([n],\B)$, where $\B = \{I \in {[n] \choose k} \ \vert \ 
p_I(C) \neq 0\}$.  Such matroids 
	are 
	called \emph{realizable} or \emph{representable}
	over $\mathbb{K}$.
\end{definition}

\begin{example}\label{ex:C}
	Consider the full rank matrix $$C=\begin{pmatrix} 
		1 & \ 0 & \ -1  & \ -2\\
		0 & \ 1 & \ 2  & \ 4
	\end{pmatrix} \hspace{.5cm}
	\text{ (or the corresponding point $C\in \Gr_{2,4}$).}
	$$
	Here $p_{12}(C)=1$, $p_{13}(C)=2$, $p_{14}(C)=4$,
	$p_{23}(C)=1$, $p_{24}(C)=2$, and $p_{34}(C)=0$.

	The corresponding matroid is  $\M(C)=\{[4], \B\}$ where 
	$\B=\{12, 13, 14, 23, 24\}$.
\end{example}

In what follows, we will 
be concerned with the \emph{real} Grassmannian
$\Gr_{k,n} = \Gr_{k,n}(\R)$.    
While every 
full rank matrix gives rise to a matroid, there are many matroids
which are \emph{not} realizable (say over $\R$), that is, they cannot 
be realized by (real) matrices.
The \emph{non-Pappus matroid} is  a 
matroid which is not realizable over any field.

The \emph{matroid stratification} of the Grassmannian is the decomposition
of $\Gr_{k,n}$ into strata consisting of all points with 
the same matroid.
While this stratification has many beautiful properties \cite{GGMS},
 we also know that by Mn\"{e}v's universality
theorem \cite{Mnev}, a matroid stratum
can have topology as bad as that of any algebraic variety!

One running theme in this article will be
that matroids and the matroid stratification of the Grassmannian
can exhibit pathological behavior, 
but when one adds the adjective ``positive'' to the 
picture, this bad behavior is replaced by the 
nicest possible statements.

\subsection{The positive Grassmannian} 
\begin{defn}\label{def:positroid}\cite{lusztig, postnikov}
We say that $V\in \Gr_{k,n}$ is \emph{totally nonnegative} 
	if (up to a global change of sign)
	$p_I(V) \geq 0$ for all $I \in {[n]\choose k}$.
Similarly, $V$ is \emph{totally positive} if $p_I(V) >0$ for all $I
	\in {[n] \choose k}$.
We let $\Grk$ and $\Gr_{k,n}^{>0}$ denote the set of 
totally nonnegative and totally positive elements of $\Gr_{k,n}$, respectively.  
$\Grk$ is called the \emph{totally nonnegative}  \emph{Grassmannian}, or 
	sometimes just the \emph{positive Grassmannian}.
\end{defn}

Note that the matrix $C$ from 
\cref{ex:C} 
represents an element of 
$\Gr_{2,4}^{\geq 0}$.

The positive and nonnegative parts of a generalized
partial flag variety $G/P$ were
first introduced by
 Lusztig \cite{lusztig}, who gave a Lie-theoretic definition
 of $(G/P)_{>0}$ and  $(G/P)_{\geq 0} :=\overline{(G/P)_{>0}}$.
Postnikov \cite{postnikov} subsequently defined 
$\Grk$
 as in \cref{def:positroid}.
These  definitions agree when $G/P=\Gr_{k,n}$
 \cite{rietsch_private},
 \cite[Corollary 1.2]{talaska_williams}.

While the positive Grassmannian was introduced rather recently, the 
theory 
of totally positive matrices is much older.  In fact one can use
results of Gantmakher and Krein
\cite{gantmakher_krein_translation} from 1950
to characterize
 $\Grk$ and $\Gr_{k,n}^{>0}$ in terms of sign variation \cite{karp}, as follows.

\begin{defn}\label{defn_var}
Given $v\in\mathbb{R}^n$, let $\var(v)$ be the number of sign changes of $v$, when $v$ is viewed as a sequence of $n$ numbers and zeros are ignored. 
	We also define
$$
\overline{\var}(v) := \max\{\var(w) : \text{$w\in\mathbb{R}^n$ such that $w_i = v_i$ for all $i\in [n]$ with $v_i\neq 0$}\},
$$
i.e.\ $\overline{\var}(v)$ is the maximum number of sign 
	changes 
	after we choose a sign for each $v_i=0$.
\end{defn}
For example, if $v := (2, 0, 2, -1)\in\mathbb{R}^4$, then $\var(v) = 1$ and $\overline{\var}(v) = 3$.

The following result is based on 
\cite[Theorems V.3, V.7, V.1, V.6]{gantmakher_krein_translation}.
\begin{thm}[{\cite[Theorem 1.1]{karp}}]
	\label{gantmakher_krein}
	Let $V\in\Gr_{k,n}$ 
	with orthogonal complement $V^{\perp}\in \Gr_{n-k,n}$.\\
(i) $V\in\Gr_{k,n}^{\ge 0}\iff\var(v)\le k-1\text{ for all }v\in V\iff\overline{\var}(w)\ge k\text{ for all }w\in V^\perp\setminus\{0\}$. \\
(ii) $V\in\Gr_{k,n}^{>0}\iff\overline{\var}(v)\le k-1\text{ for all }v\in V\setminus\{0\}\iff\var(w)\ge k\text{ for all }w\in V^\perp\setminus\{0\}$.
\end{thm}

\subsection{The positroid cell decomposition}

Despite the fact that the topology of matroid strata can be very bad,
Postnikov realized that if
one intersects these strata with the positive Grassmannian,
one obtains a \emph{cell decomposition} \cite{postnikov}.
In fact, it is a regular CW decomposition \cite{PSW, Williams, RW, PKL}.

\begin{theorem} \cite{postnikov}
For $\M\subseteq \binom{[n]}{k}$, let 
$$S_{\M}:= \{V \in \Grk \ \vert \ p_I(V)>0 \text{ if and only if }
I\in \M\}.$$
	Then $\Grk = \cup S_{\M}$ is a cell decomposition, i.e. 
	each $S_{\M}$ is an open ball.

If $S_\M\neq\emptyset$, we call $\M$ a \emph{positroid} and $S_\M$ its \emph{positroid cell}.
\end{theorem}

More generally, 
Rietsch  
gave
a cell decomposition of $(G/P)_{\geq 0}$ \cite{rietsch}.
When $G/P = \Gr_{k,n}$, the two cell decompositions agree 
	\cite[Corollary 1.2]{talaska_williams}.


As shown in \cite{postnikov} and explained below,
the cells of $\Grk$
can be indexed by 
  combinatorial objects such as
	\emph{decorated permutations} $\pi$
	or move-equivalence classes of \emph{plabic graphs} $G$,
	see e.g.\ \cite[Chapter 7]{FWZ}.  We will correspondingly 
	 refer to such cells as $S_{\pi}$ and $S_G$.


\begin{defn}\label{defn:decperm}
A \emph{decorated permutation} on $[n]$ is a permutation $\pi\in S_n$
 whose fixed points are each coloured either black (``loop'') or white (``coloop''). We denote a black fixed point $i$ by $\pi(i) = \underline{i}$, and a white fixed point $i$ by $\pi(i) = \overline{i}$.
An \emph{anti-excedance} of a decorated permutation $\pi$ 
is an element $i \in [n]$ such that either $\pi^{-1}(i) > i$ or $\pi(i)=\overline{i}$.  
We say that a decorated permutation
       on $[n]$ is of
       \emph{type $(k,n)$} if it has $k$ anti-excedances.
\end{defn}

For example, $\pi = (3,\underline{2},5,1,6,8,\overline{7},4)$ has
a loop in position $2$ and a coloop in position $7$.
Its anti-excedances are $1$, $4$, and $7$.

	\begin{definition}
To a $k\times n$ matrix $C$ 
	with columns $(c^1,\dots,c^n)$ representing 
 an element of $\Grk$,
we associate a decorated 
		permutation $\pi:=\pi_C$ of type $(k,n)$
		as follows.
We set $\pi(i):=j$ to be the label of the first column $j$ such that
$c^i \in \spn \{c^{i+1}, c^{i+2},\dots, c^j\}$, where the columns
are listed in cyclic order (going from $c^n$ to $c^1$ if $i+1>j$).
If $c^i=\mathbf{0}$, then  
$i$ is a \emph{loop} of matroid $\M(C)$ 
and we set $\pi(i)=\underline{i}$,
and if $c^i$ is not in the span of the other column
vectors,  then $i$ is a \emph{coloop} of $\M(C)$ and we set $\pi(i)=\overline{i}$. 
\end{definition}
One can show that the construction above gives a well-defined function 
from $\Grk$ to decorated permutations of type
$(k,n)$.  
If $C$ is the matrix from \cref{ex:C},
then $\pi_C=(3,1,4,2)$.

\begin{proposition}
Let $\pi$ be a decorated permutation  of type $(k,n)$, and
	let $$S_{\pi} = \{C \in \Grk \ \vert \ \pi_C = \pi\}.$$
	Then $S_{\pi}$ is a positroid cell, and all positroid cells
	of $\Grk$ have the form $S_{\pi}$ for some decorated permutation 
	$\pi$ of type 
	$(k,n)$.
\end{proposition}

\begin{figure}[h]\centering
	\resizebox{1.5in}{!}{
	\begin{tikzpicture}
		\draw (0,0) circle (2.15cm);
	\node[draw=none, minimum size=4.3cm, regular polygon, regular polygon sides=9] (s) {};
	\node[draw=none, minimum size=3.3cm, regular polygon, regular polygon sides=9] (a) {};
	\foreach \x in {1,2,...,9}
	\fill (s.corner \x) circle[radius=1.2pt];
	\foreach \x in {1,2,...,9}
	\fill (a.corner \x) circle[radius=2pt];
	\foreach \x in {1,2,...,9}
        \draw[black,thick](a.corner \x)--(s.corner \x);
	 \node[shift=(s.corner 1), anchor=south] {$1$};
	 \node[shift=(s.corner 2),anchor=south] {$9$};
	 \node[shift=(s.corner 3),anchor=east] {$8$};
	 \node[shift=(s.corner 4),anchor=east] {$7$};
	 \node[shift=(s.corner 5),anchor=north] {$6$};
	 \node[shift=(s.corner 6),anchor=north] {$5$};
	 \node[shift=(s.corner 7),anchor=west] {$4$};
	 \node[shift=(s.corner 8),anchor=west] {$3$};
	 \node[shift=(s.corner 9),anchor=west] {$2$};
        \draw[black,thick](-.8,.8)--(a.corner 1);
        \draw[black,thick](-.8,.8)--(a.corner 2);
        \draw[black,thick](-.8,.8)--(a.corner 4);
        \draw[black,thick](-1.4,.3)--(a.corner 2);
        \draw[black,thick](-1.4,.3)--(a.corner 3);
        \draw[black,thick](-1.4,.3)--(a.corner 4);
        \draw[black,thick](0.5,.3)--(a.corner 4);
        \draw[black,thick](0.5,.3)--(a.corner 9);
        \draw[black,thick](0.5,.3)--(a.corner 8);
        \draw[black,thick](0.3,-1.1)--(a.corner 4);
        \draw[black,thick](0.3,-1.1)--(a.corner 6);
        \draw[black,thick](0.3,-1.1)--(a.corner 7);
        \draw[black,thick](0.8,-.4)--(a.corner 4);
        \draw[black,thick](0.8,-.4)--(a.corner 7);
        \draw[black,thick](0.8,-.4)--(a.corner 8);
	\filldraw[color=black,fill=white] (-.8,.8) circle (2pt);
	\filldraw[color=black,fill=white] (-1.4,.3) circle (2pt);
	\filldraw[color=black,fill=white] (0.5,.3) circle (2pt);
	\filldraw[color=black,fill=white] (0.3,-1.1) circle (2pt);
	\filldraw[color=black,fill=white] (0.8,-.4) circle (2pt);
\end{tikzpicture}
	}
	\caption{A plabic graph $G$ with  
	$\pi_G = (8,5,9,2,3,\underline{6},4,1,7)$. It has
	a black lollipop at $6$.}
        \label{fig:plabic}
\end{figure}
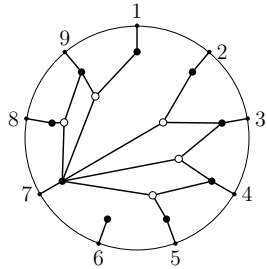
\begin{defn}\label{def:plabic}
A \emph{planar bicolored graph} 
(or 
\emph{plabic graph})
is a planar graph $G$ properly 
embedded into a closed disk
                with (uncolored) vertices lying on the
                   boundary of the disk labeled $1,\dots, n$
		   in clockwise order
                       for some positive $n$, such that: 
                   each  boundary vertex is incident to a single
                        edge;
                        each internal vertex is colored black or white;
                        and each internal vertex is connected by
                        a path to some boundary vertex.
See \cref{fig:plabic}.

	If a boundary vertex $i$ is attached to an edge whose other endpoint
	is a leaf, we call this component a \emph{lollipop}.
	\emph{We will assume that $G$ has no internal leaves except for lollipops}.
\end{defn}

We next describe some local moves on plabic graphs, see
 \cref{fig:M1}.

(M1) Square Move.  If there is a square formed by
four trivalent vertices whose colors alternate,
then we can switch the
colors of these four vertices.

(M2) 
Two adjacent internal vertices of the same color can be merged.
Alternatively, we can split an internal vertex into two vertices of the 
same color joined by an edge.

(M3) 
We can remove/add degree $2$ vertices, as shown.

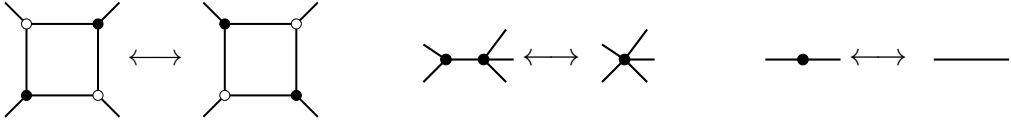
\begin{figure}[h]
\resizebox{1.7in}{!}{
\begin{tikzpicture}
	\draw[black, thick] (0,0) rectangle (1,1);
	\draw[black,thick](0,0)--(-.3,-.3);
	\draw[black,thick](0,1)--(-.3,1.3);
	\draw[black,thick](1,1)--(1.3,1.3);
	\draw[black,thick](1,0)--(1.3,-.3);
	\filldraw[color=black,fill=black] (0,0) circle (2pt);
	\filldraw[color=black,fill=black] (1,1) circle (2pt);
	\filldraw[color=black,fill=white] (0,1) circle (2pt);
	\filldraw[color=black,fill=white] (1,0) circle (2pt);
	\draw (1.8, .5) node {$\longleftrightarrow$};
\end{tikzpicture}
\begin{tikzpicture}
	\draw[black, thick] (0,0) rectangle (1,1);
	\draw[black,thick](0,0)--(-.3,-.3);
	\draw[black,thick](0,1)--(-.3,1.3);
	\draw[black,thick](1,1)--(1.3,1.3);
	\draw[black,thick](1,0)--(1.3,-.3);
	\filldraw[color=black,fill=white] (0,0) circle (2pt);
	\filldraw[color=black,fill=white] (1,1) circle (2pt);
	\filldraw[color=black,fill=black] (0,1) circle (2pt);
	\filldraw[color=black,fill=black] (1,0) circle (2pt);
\end{tikzpicture}}
\quad \quad \quad
\begin{tikzpicture}
	\draw[black,thick](-.3,.7)--(0,.5)--(.5,.5)--(.8,.9);
	\draw[black,thick](-.3,.2)--(0,.5);
	\draw[black,thick](.9,.5)--(.5,.5)--(.8,.2);
	\filldraw[color=black,fill=black] (0,.5) circle (2pt);
	\filldraw[color=black,fill=black] (.5,.5) circle (2pt);
	\filldraw[color=white,fill=white] (0,-.2) circle (2pt);
	\draw (1.4, .5) node {$\longleftrightarrow$};
\end{tikzpicture}
\begin{tikzpicture}
	\draw[black,thick](-.3,.7)--(0,.5)--(.3,.9);
	\draw[black,thick](-.3,.2)--(0,.5);
	\draw[black,thick](.4,.5)--(0,.5)--(.3,.2);
	\filldraw[color=black,fill=black] (0,.5) circle (2pt);
	\filldraw[color=white,fill=white] (0,-.2) circle (2pt);
\end{tikzpicture} \quad \quad \quad
\begin{tikzpicture}
	\draw[black,thick](0,0)--(1,0);
	\filldraw[color=black,fill=black] (.5,0) circle (2pt);
	\filldraw[color=white,fill=white] (0,-.7) circle (2pt);
	\draw (1.5, 0) node {$\longleftrightarrow$};
\end{tikzpicture}
\begin{tikzpicture}
	\draw[black,thick](0,0)--(1,0);
	\filldraw[color=white,fill=white] (0,-.7) circle (2pt);
\end{tikzpicture}
\caption{Moves (M1), (M2), (M3) on plabic graphs.}
\label{fig:M1}
\end{figure}

\begin{definition}\label{def:move}
Two plabic graphs are \emph{move-equivalent} if they can be obtained
from each other by moves (M1)-(M3).  
A plabic graph 
is \emph{reduced} if there is no graph move-equivalent to it
	in which 
 two adjacent vertices $u$ and $v$ are connected by more
	than one edge.
\end{definition}

	
\cref{def:rules} and 
\cref{prop:matching} give several ways 
to read a positroid off of a plabic graph.  The positroid depends only on 
the move-equivalence class of the plabic graph.

	\begin{defn}\label{def:rules}
		Let $G$ be a reduced plabic graph as above with boundary vertices $1,\dots, n$. For each boundary vertex $i\in [n]$, we follow a path along the edges of $G$ starting at $i$, turning (maximally) right at every internal black vertex, and (maximally) left at every internal white vertex. This path ends at some boundary vertex $\pi(i)$. By \cite[Section 13]{postnikov}, the fact that $G$ is reduced implies that each fixed point of $\pi$ is attached to a lollipop; we color each fixed point by the color of its lollipop. In this way we obtain the \emph{(decorated) trip permutation} $\pi_G = \pi$ of $G$. We say that $G$ is of \emph{type} $(k,n)$, where $k$ is the number of anti-excedances of $\pi_G$.
\end{defn}
In \cref{fig:plabic}
we have $\pi_G = 
	 (8,5,9,2,3,\underline{6},4,1,7)$, which has
	 $k= 5$ anti-excedances.

	\begin{theorem}[Fundamental theorem of reduced plabic graphs, 
		{\cite[Theorem 13.4]{postnikov}, 
		see also {\cite[Theorem 7.4.25]{FWZ}}}]
Let $G$ and $G'$ be reduced plabic graphs.  Then $G$ and $G'$ are 
move-equivalent if and only if $G$ and $G'$ have the same decorated
trip permutation.
	\end{theorem}


\begin{definition}\label{def:matching}
Let $G$ be a bipartite plabic graph in which each 
boundary vertex
	is incident to a white vertex.  
	An \emph{almost perfect matching} of $G$
is a
subset $M$ of edges such that each internal vertex is incident to
exactly one edge in $M$ (and each boundary vertex $i$ is incident
to either one or no edges in $M$).
        We let $\partial M = \{i \ \vert \ i \text{ is incident to an edge of }M\}$.
\end{definition}

Given a plabic graph, we can use move (M3) to ensure
that the resulting graph is bipartite and that 
 each boundary vertex is incident to a 
white vertex. 
(Note that we can think of such a graph as a 
bipartite graph $G$
 in which all boundary vertices are colored black.)

\begin{proposition}
	[{\cite[Proposition~11.7, Lemma~11.10]{postnikov}}]
	\label{prop:matching}
    Let $G$ be a bipartite plabic graph such that each 
boundary vertex is incident to a white vertex. 
Let $$\M(G) = \{\partial M \ \vert \ M \text{ an almost perfect matching of }G\}.$$
	If $\M(G)$ is nonempty,  then
	$\M(G)$ is the set of  bases of a positroid on $[n]$.
	Moreover, all positroids arise from plabic graphs.
\end{proposition}

See  \cref{Plabic1} for an example.

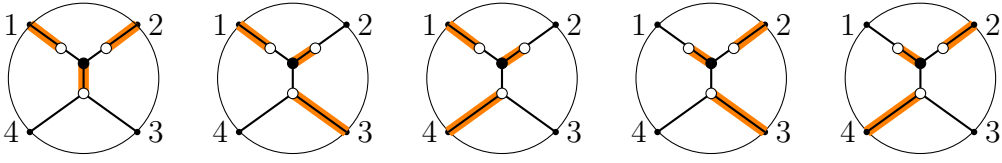
\begin{figure}[h]\centering
	\begin{tikzpicture}
	\node[draw=none, minimum size=2cm, regular polygon, regular polygon sides=4] (s) {};
	\foreach \x in {1,2,...,4}
	\fill (s.corner \x) circle[radius=1.2pt]; 
        \draw[orange, line width= 4pt](-.3,.4)--(s.corner 2);
        \draw[orange, line width= 4pt](.3,.4)--(s.corner 1);
        \draw[orange, line width= 4pt](0,.2)--(0,-.2);
	\draw (0,0) circle (1cm);
	 \node[shift=(s.corner 1),anchor=west] {$2$};
	 \node[shift=(s.corner 2),anchor=east] {$1$};
	 \node[shift=(s.corner 3),anchor=east] {$4$};
	 \node[shift=(s.corner 4),anchor=west] {$3$};
        \draw[black,thick](0,-.2)--(s.corner 3);
        \draw[black,thick](0,-.2)--(s.corner 4);
        \draw[black,thick](0,.2)--(s.corner 1);
        \draw[black, thick](0,.2)--(s.corner 2);
        \draw[black,thick](0,.2)--(0,-.2);
	\filldraw[color=black,fill=black] (0,.2) circle (2pt);
	\filldraw[color=black,fill=white] (0,-.2) circle (2pt);
	\filldraw[color=black,fill=white] (-.3,.4) circle (2pt);
	\filldraw[color=black,fill=white] (.3,.4) circle (2pt);
	\end{tikzpicture} 
	\hspace{.1cm} 
	\begin{tikzpicture}
		\draw (0,0) circle (1cm);
	\node[draw=none, minimum size=2cm, regular polygon, regular polygon sides=4] (s) {};
	\foreach \x in {1,2,...,4}
	\fill (s.corner \x) circle[radius=1.2pt];
        \draw[orange, line width= 4pt](-.3,.4)--(s.corner 2);
        \draw[orange, line width= 4pt](0,-.2)--(s.corner 4);
        \draw[orange, line width= 4pt](0,.2)--(.3,.4);
	 \node[shift=(s.corner 1),anchor=west] {$2$};
	 \node[shift=(s.corner 2),anchor=east] {$1$};
	 \node[shift=(s.corner 3),anchor=east] {$4$};
	 \node[shift=(s.corner 4),anchor=west] {$3$};
        \draw[black,thick](0,-.2)--(s.corner 3);
        \draw[black,thick](0,-.2)--(s.corner 4);
        \draw[black,thick](0,.2)--(s.corner 1);
        \draw[black,thick](0,.2)--(s.corner 2);
        \draw[black,thick](0,.2)--(0,-.2);
	\filldraw[color=black,fill=black] (0,.2) circle (2pt);
	\filldraw[color=black,fill=white] (0,-.2) circle (2pt);
	\filldraw[color=black,fill=white] (-.3,.4) circle (2pt);
	\filldraw[color=black,fill=white] (.3,.4) circle (2pt);
\end{tikzpicture} 
	\hspace{.1cm} 
	\begin{tikzpicture}
		\draw (0,0) circle (1cm);
	\node[draw=none, minimum size=2cm, regular polygon, regular polygon sides=4] (s) {};
	\foreach \x in {1,2,...,4}
	\fill (s.corner \x) circle[radius=1.2pt];
        \draw[orange, line width= 4pt](-.3,.4)--(s.corner 2);
        \draw[orange, line width= 4pt](0,-.2)--(s.corner 3);
        \draw[orange, line width= 4pt](0,.2)--(.3,.4);
	 \node[shift=(s.corner 1),anchor=west] {$2$};
	 \node[shift=(s.corner 2),anchor=east] {$1$};
	 \node[shift=(s.corner 3),anchor=east] {$4$};
	 \node[shift=(s.corner 4),anchor=west] {$3$};
        \draw[black,thick](0,-.2)--(s.corner 3);
        \draw[black,thick](0,-.2)--(s.corner 4);
        \draw[black,thick](0,.2)--(s.corner 1);
        \draw[black,thick](0,.2)--(s.corner 2);
        \draw[black,thick](0,.2)--(0,-.2);
	\filldraw[color=black,fill=black] (0,.2) circle (2pt);
	\filldraw[color=black,fill=white] (0,-.2) circle (2pt);
	\filldraw[color=black,fill=white] (-.3,.4) circle (2pt);
	\filldraw[color=black,fill=white] (.3,.4) circle (2pt);
\end{tikzpicture}  
	\hspace{.1cm} 
	\begin{tikzpicture}
		\draw (0,0) circle (1cm);
	\node[draw=none, minimum size=2cm, regular polygon, regular polygon sides=4] (s) {};
	\foreach \x in {1,2,...,4}
	\fill (s.corner \x) circle[radius=1.2pt];
        \draw[orange, line width= 4pt](.3,.4)--(s.corner 1);
        \draw[orange, line width= 4pt](0,-.2)--(s.corner 4);
        \draw[orange, line width= 4pt](0,.2)--(-.3,.4);
	 \node[shift=(s.corner 1),anchor=west] {$2$};
	 \node[shift=(s.corner 2),anchor=east] {$1$};
	 \node[shift=(s.corner 3),anchor=east] {$4$};
	 \node[shift=(s.corner 4),anchor=west] {$3$};
        \draw[black,thick](0,-.2)--(s.corner 3);
        \draw[black,thick](0,-.2)--(s.corner 4);
        \draw[black,thick](0,.2)--(s.corner 1);
        \draw[black,thick](0,.2)--(s.corner 2);
        \draw[black,thick](0,.2)--(0,-.2);
	\filldraw[color=black,fill=black] (0,.2) circle (2pt);
	\filldraw[color=black,fill=white] (0,-.2) circle (2pt);
	\filldraw[color=black,fill=white] (-.3,.4) circle (2pt);
	\filldraw[color=black,fill=white] (.3,.4) circle (2pt);
\end{tikzpicture} 
	\hspace{.1cm} 
	\begin{tikzpicture}
		\draw (0,0) circle (1cm);
	\node[draw=none, minimum size=2cm, regular polygon, regular polygon sides=4] (s) {};
        \draw[orange, line width= 4pt](.3,.4)--(s.corner 1);
        \draw[orange, line width= 4pt](0,-.2)--(s.corner 3);
        \draw[orange, line width= 4pt](0,.2)--(-.3,.4);
	\foreach \x in {1,2,...,4}
	\fill (s.corner \x) circle[radius=1.2pt];
	 \node[shift=(s.corner 1),anchor=west] {$2$};
	 \node[shift=(s.corner 2),anchor=east] {$1$};
	 \node[shift=(s.corner 3),anchor=east] {$4$};
	 \node[shift=(s.corner 4),anchor=west] {$3$};
        \draw[black,thick](0,-.2)--(s.corner 3);
        \draw[black,thick](0,-.2)--(s.corner 4);
        \draw[black,thick](0,.2)--(s.corner 1);
        \draw[black,thick](0,.2)--(s.corner 2);
        \draw[black,thick](0,.2)--(0,-.2);
	\filldraw[color=black,fill=black] (0,.2) circle (2pt);
	\filldraw[color=black,fill=white] (0,-.2) circle (2pt);
	\filldraw[color=black,fill=white] (-.3,.4) circle (2pt);
	\filldraw[color=black,fill=white] (.3,.4) circle (2pt);
\end{tikzpicture} 
	\caption{A bipartite plabic graph $G_1$ with 
	 $\pi_{G_1}=(3,1,4,2)$
	which has 
	five almost-perfect matchings. The corresponding positroid is 
	$([4], \M(G_1))$ where 
	$\M(G_1) = \{12, 13, 14, 23, 24\}$. }
        \label{Plabic1}
\end{figure}


Postnikov used plabic graphs to give parameterizations of cells
of $\Grk$ \cite{postnikov}; this result
can be recast in terms of \emph{flows} \cite{Talaska} or 
as a variant of a theorem of Kasteleyn
 \cite{Speyer}. 
\begin{theorem}[\cite{postnikov, Talaska, Speyer}]\label{thm:Kasteleyn}
Let $G$ be a bipartite plabic graph with 
 $n$ boundary vertices, all of which are colored black.
Suppose $G$ has at least one almost 
 perfect matching $M_0$, and let  $k = |\partial M_0|$.  
Let $w:\Edges(G) \to \R_{>0}$ be any weight function,
and for $M$  an almost perfect matching,
let $w(M) := \prod_{e\in M} w_e$, where 
	 $w_{e}$ denotes the weight 
of edge $e$.
Then there is a 
	$k\times n$ 
matrix 	$L=L(w)$ representing a point of 
	$\Grk$ 
	such that
	$$p_I(L) = 
	\sum_{M: \partial M = I} w(M) \hspace{.2cm} \text{ for all }I \in {[n] \choose k}.$$
	Moreover, if we let $w$ vary over weight functions, we obtain
	a positroid cell
	$$S_G:=\{L(w) \ \vert \ w:\Edges(G)\to \R_{>0}\}.$$
\end{theorem}

If $G$ is a tree, we call $S_G$  a
\emph{tree positroid cell}.


\begin{remark}
If $G$ 
is a plabic graph as in \cref{thm:Kasteleyn} which is reduced, with decorated permutation $\pi_G$
and almost perfect matchings $\M(G)$, we have
	that $S_G = S_{\M(G)} = S_{\pi_G}$ \cite{postnikov}. 
	So we can  index
	positroid cells
 by plabic graphs, bases, or 
	decorated permutations.
\end{remark}

\subsection{$\phi$-induced subdivisions and positroid tilings}

Given a surjective map $\phi: X \to Y$ from a cell complex $X$
onto a topological
space $Y$, 
it is natural to try to decompose $Y$ using images
of 
 cells under $\phi$.

\begin{definition}\label{def:tri0}
	Let $X = \bigsqcup_{\pi} S_{\pi}$ be a cell complex
and let $\phi: X \to Y$ be
 a continuous surjective map onto $Y$, 
a $d$-dimensional cell complex or subset thereof.
We define a \emph{$\phi$-induced dissection of $Y$}
to be a collection 
$\{\overline{\phi(S_{\pi})} \ \vert \ \pi\in \mathcal{C}\}$
of images of cells of $X$, indexed by the set 
$\mathcal{C}$, such that:
		their union $\cup_{\pi\in \CCC} \overline{\phi(S_{\pi})}$
		equals $Y$;
		the interiors are pairwise disjoint, i.e. 
		$\phi(S_{\pi}) \cap \phi(S_{\pi'}) = \emptyset$
		for $\pi \neq \pi' \in \CCC$; and 
		$\dim(\phi(S_{\pi})) = d$ for all $\pi \in \CCC$; 

We call a dissection
$\{\overline{\phi(S_{\pi})} \ \vert \ \pi\in \mathcal{C}\}$
	a \emph{$\phi$-induced tiling} if additionally 
	$\phi$ is injective on each $S_{\pi}$ for $\pi \in \CCC$.
And we call a dissection
$\{\overline{\phi(S_{\pi})} \ \vert \ \pi\in \mathcal{C}\}$
	a \emph{$\phi$-induced subdivision} if whenever
	 $\overline{\phi(S_{\pi})} \cap \overline{\phi(S_{\pi'})} \neq \emptyset$, this intersection equals $\overline{\phi(S_{\pi''})}$, where 
	 $S_{\pi''}$ lies in $\overline{S_{\pi}} \cap \overline{S_{\pi'}}$.
\end{definition}

When $\phi:X \to Y$ is an affine projection of convex polytopes, 
the above notion of $\phi$-induced subdivision recovers 
Billera-Sturmfels' notion of $\phi$-induced polyhedral subdivision \cite{BS}.
The subdivisions which are also $\phi$-induced tilings
are their \emph{tight} $\phi$-induced subdivisions.
The relation to polyhedral subdivisions suggests a number of 
 questions.  What can one say about the
\emph{Baues poset} $\omega(\phi:X\to Y)$ of $\phi$-induced dissections
or subdivisions, partially ordered by refinement?  What can one
say about the \emph{flip graph}, the restriction of the 
Hasse diagram of $\omega(\phi:X\to Y)$ to elements of rank $0$ and $1$? When is 
it connected?  Is there an analogue of 
 \emph{fiber polytopes} \cite{BS} (perhaps along the lines of 
 \cite{MathisMeroni}), which 
control the \emph{coherent} $\pi$-induced subdivisions?  

\begin{definition}\label{def:tri}
In \cref{def:tri0}, let $X$ be  the positive Grassmannian
$\Grk$ with its positroid cell decomposition.
For $S_\pi$ a $d$-dimensional positroid cell,
we say that $\overline{\phi(S_{\pi})}$
is a \emph{positroid tile} (for $\phi$)
if $\phi$ is injective on $S_{\pi}$.
We will refer to $\phi$-induced dissections, tilings, and subdivisions,
	as \emph{positroid dissections, tilings}, and \emph{subdivisions}.
Our notion of positroid subdivision is closely related to the
	\emph{good dissections} studied in \cite{LPW}.
\end{definition}

In this article we will take $X$ to be the positive Grassmannian, 
and consider the case where
 $\phi$ is the \emph{moment map} 
(\cref{sec:moment}) or
the   \emph{amplituhedron map} 
(\cref{sec:amp}).

\section{The moment map, positroid tilings, and the positive tropical Grassmannian}
\label{sec:moment}

The foundational 1987 paper
 of Gelfand--Goresky--MacPherson--Serganova \cite{GGMS}
initiated the study of the Grassmannian and its matroid 
stratification via the moment map.
Here we will consider the restriction of the moment map
to the positive Grassmannian.

Given a subset $I \subset [n]$ and a point $x\in \R^n$,
we use the notation $x_I:=\sum_{i \in I} x_i$.
We also let $e_I := \sum_{i \in I} e_i \in \R^n$, where
$\{e_1, \dotsc, e_n\}$ is the standard basis of $\RR^n$.


The torus $T = (\CC^{*})^n$ acts on $\Gr_{k,n}$ by scaling the columns of
a matrix representative $C$.  
This torus action gives rise to a 
 \emph{moment map}  $\mu:\Gr_{k,n} \to \R^n$.
\begin{definition}
        Let $C 
        \in \Gr_{k,n}$.
        The \emph{moment map}
        $\mu: \Gr_{k,n} \to \R^n$ is defined by
        $$\mu(C) = \frac{ \sum_{I \in \binom{[n]}{k}} |p_I(C)|^2 e_I}
        {\sum_{I \in \binom{[n]}{k}} |p_I(C)|^2}.$$
\end{definition}


Let $TC$ denote the orbit of $C$ under the action of $T$, and
$\overline{TC}$ its closure.
It follows from 
 \cite{GS} 
 that the image
$\mu(\overline{TC})$ is a convex polytope, whose vertices
are the images of the torus-fixed points. This polytope
is the \emph{matroid polytope} $\Gamma_{\M(C)}$ \cite{GGMS}, as defined below.

\begin{defn}
Given a matroid $\M=([n],\B)$, the (basis) \emph{matroid polytope}
$\Gamma_{\M}$ is the convex hull of the indicator vectors of the bases
of $\M$: 
$$\Gamma_{\M} := \convex \{e_B \ \vert \ B\in \B\}\subset \R^n.$$
\end{defn}

A special case of a matroid polytope is 
the \emph{hypersimplex}
$\Delta_{k,n}$, 
 the convex hull of all points $e_I$ for $I \in \binom{[n]}{k}$.
We have that $\mu(\Gr_{k,n}) = \Delta_{k,n}$.


\subsection{Classification of positroid tiles for the moment map}

Now let us consider the positive analogues of some of the above objects.
If $\M$ is  a \emph{positroid}, the matroid polytope $\Gamma_{\M}$
is called
 a \emph{positroid polytope}.  
If one restricts the moment map to 
$\Grk$, one can show that the moment map image 
$\mu(\overline{S_{\M}}) = \overline{\mu(S_{\M})}$
of $S_{\M}$ is precisely the corresponding positroid polytope
$\Gamma_{\M}$
\cite[Proposition 7.10]{tsukerman_williams}.
In particular, the moment map image of $\Gr_{k,n}^{\geq 0}$ 
is again the hypersimplex
$\Delta_{k,n}$. 
If $\M$ is the positroid associated to a cell 
$S_\pi$ or $S_G$, we also use the notation $\Gamma_{\pi}$ and $\Gamma_G$
to refer to $\Gamma_{\M}$.
See	\cref{fig:pospolytopes}.

\begin{figure}[h]\centering
	\begin{tikzpicture} 
		\draw (0,0) circle (1cm);
	\node[draw=none, minimum size=2cm, regular polygon, regular polygon sides=4] (s) {};
	\foreach \x in {1,2,...,4}
	\fill (s.corner \x) circle[radius=1.2pt];
	 \node[shift=(s.corner 1),anchor=west] {$2$};
	 \node[shift=(s.corner 2),anchor=east] {$1$};
	 \node[shift=(s.corner 3),anchor=east] {$4$};
	 \node[shift=(s.corner 4),anchor=west] {$3$};
        \draw[black,thick](0,-.2)--(s.corner 3);
        \draw[black,thick](0,-.2)--(s.corner 4);
        \draw[black,thick](0,.2)--(s.corner 1);
        \draw[black,thick](0,.2)--(s.corner 2);
        \draw[black,thick](0,.2)--(0,-.2);
	\filldraw[color=black,fill=black] (0,.2) circle (2pt);
	\filldraw[color=black,fill=white] (0,-.2) circle (2pt);
	\filldraw[color=black,fill=white] (-.3,.4) circle (2pt);
	\filldraw[color=black,fill=white] (.3,.4) circle (2pt);
	\draw (0,-1.2) node {\tiny $G_1$};
\end{tikzpicture} \quad 
	\begin{tikzpicture}
     \tkzDefPoint(0,0){e14}
     \tkzDefPoint(1,0){e13}
     \tkzDefPoint(.5,.4){e24}
     \tkzDefPoint(1.5,.4){e23}
     \tkzDefPoint(.7,1.4){e12}
     \tkzDefPoint(.7,-1){e34}
		\tkzDrawPolygon[dashed](e24, e14,e13,e23,e24,e12,e14,e13,e12,e23,e34,e14,e13,e34,e23)
		\tkzDrawPolygon(e24,e14,e12,e24,e23,e12,e13,e14,e13,e23);
		\node[shift=(e24), anchor=north] {\tiny $e_{24}$};
		\node[shift=(e14), anchor= east] {\tiny $e_{14}$};
		\node[shift=(e23), anchor=west] {\tiny $e_{23}$};
		\node[shift=(e13), anchor=north east] {\tiny $e_{13}$};
		\node[shift=(e12), anchor=east] {\tiny $e_{12}$};
		\node[shift=(e34), anchor=east] {\tiny $e_{34}$};
	\filldraw[color=black,fill=black] (e24) circle (1pt);
	\filldraw[color=black,fill=black] (e14) circle (1pt);
	\filldraw[color=black,fill=black] (e23) circle (1pt);
	\filldraw[color=black,fill=black] (e13) circle (1pt);
	\filldraw[color=black,fill=black] (e12) circle (1pt);
	\filldraw[color=black,fill=black] (e34) circle (1pt);
\end{tikzpicture} \quad \quad
	\begin{tikzpicture}
     \tkzDefPoint(0,0){e14}
     \tkzDefPoint(1,0){e13}
     \tkzDefPoint(.5,.4){e24}
     \tkzDefPoint(1.5,.4){e23}
     \tkzDefPoint(.7,1.4){e12}
     \tkzDefPoint(.7,-1){e34}
\tkzDrawPolygon[dashed](e14,e12,e13,e23,e12,e24,e14)
\tkzDrawPolygon(e34,e23,e24,e14,e34,e13,e23,e24,e14,e13);
		\node[shift=(e24), anchor=north] {\tiny $e_{24}$};
		\node[shift=(e14), anchor= east] {\tiny $e_{14}$};
		\node[shift=(e23), anchor=west] {\tiny $e_{23}$};
		\node[shift=(e13), anchor=north east] {\tiny $e_{13}$};
		\node[shift=(e12), anchor=east] {\tiny $e_{12}$};
		\node[shift=(e34), anchor=east] {\tiny $e_{34}$};
	\filldraw[color=black,fill=black] (e24) circle (1pt);
	\filldraw[color=black,fill=black] (e14) circle (1pt);
	\filldraw[color=black,fill=black] (e23) circle (1pt);
	\filldraw[color=black,fill=black] (e13) circle (1pt);
	\filldraw[color=black,fill=black] (e12) circle (1pt);
	\filldraw[color=black,fill=black] (e34) circle (1pt);
\end{tikzpicture}
	\begin{tikzpicture}
		\draw (0,0) circle (1cm);
	\node[draw=none, minimum size=2cm, regular polygon, regular polygon sides=4] (s) {};
	\foreach \x in {1,2,...,4}
	\fill (s.corner \x) circle[radius=1.2pt];
	 \node[shift=(s.corner 1),anchor=west] {$2$};
	 \node[shift=(s.corner 2),anchor=east] {$1$};
	 \node[shift=(s.corner 3),anchor=east] {$4$};
	 \node[shift=(s.corner 4),anchor=west] {$3$};
        \draw[black,thick](0,-.2)--(s.corner 3);
        \draw[black,thick](0,-.2)--(s.corner 4);
        \draw[black,thick](0,.2)--(s.corner 1);
        \draw[black,thick](0,.2)--(s.corner 2);
        \draw[black,thick](0,.2)--(0,-.2);
	\filldraw[color=black,fill=white] (0,.2) circle (2pt);
	\filldraw[color=black,fill=black] (0,-.2) circle (2pt);
	\filldraw[color=black,fill=white] (-.3,-.4) circle (2pt);
	\filldraw[color=black,fill=white] (.3,-.4) circle (2pt);
	\draw (0,-1.2) node {\tiny $G_2$};
\end{tikzpicture}
	\caption{Positroid polytopes $\Gamma_{G_1}$ and $\Gamma_{G_2}$
	associated to  graphs $G_1$ and $G_2$, cf.
	 \cref{Plabic1}.} 
        \label{fig:pospolytopes}
\end{figure}



Applying \cref{def:tri} to the moment map 
$\mu: \Grk\to \Delta_{k,n}$ onto the hypersimplex, a polytope of dimension $n-1$,
we see that a positroid tile is the (closure of the) image of an $(n-1)$-dimensional positroid
cell
on which the moment map is injective.

\begin{proposition} [{{\cite[Proposition 3.16]{LPW}}, based on 
	\cite{tsukerman_williams}, \cite{ARW}}]
\label{prop:tree}
The positroid tiles for the moment map are exactly 
the positroid polytopes $\Gamma_G$ associated to 
the tree positroid cells $S_G$. 
	Two positroid tiles $\Gamma_G$ and $\Gamma_{G'}$ are the same  if and only if
	$G$ and $G'$ are related by move (M2).
\end{proposition}

\subsection{The positive tropical Grassmannian and positroid subdivisions}

How can one produce positroid tilings and more generally positroid
subdivisions of the hypersimplex $\Delta_{k,n}$?  One way is to use the 
\emph{positive tropical Grassmannian}.

The \emph{tropical Grassmannian} 
$\Trop \Gr_{k,n}$ \cite{HKT, KT, tropgrass}
is the space
of \emph{realizable tropical linear spaces}, obtained by applying the valuation map to 
elements of the Grassmannian $\Gr_{k,n}(K)$ over the field 
$K=\CC\{\!\{t\}\! \}$
of 
\emph{Puiseux-series}. 
Meanwhile the \emph{Dressian} $\Dr_{k,n}$ 
is the space of \emph{tropical Pl\"ucker vectors}
$P = \{P_I\}_{I \in {[n] \choose k}}$, also known 
as 
\emph{valuated matroids}.  Thinking of each 
$P\in \Dr_{k,n}$ as a \emph{height
function} on the vertices of the hypersimplex $\Delta_{k,n}$, one can show that the Dressian
parameterizes
regular matroid subdivisions $\mathcal{D}_P$
of $\Delta_{k,n}$  \cite{Kapranov, Speyer}, which in turn are dual to 
\emph{abstract tropical linear spaces} \cite{Speyer}.

There are positive notions of both of the above spaces.
The \emph{positive tropical Grassmannian} \cite{troppos} is the space
of \emph{realizable positive tropical linear spaces}, obtained by applying the valuation map to
Puiseux-series valued elements of the {positive Grassmannian}.
The \emph{positive Dressian} is the space of \emph{positive tropical Pl\"ucker vectors}.


\begin{definition}\label{def:tropPlucker}
We say
	$P=\{P_I\}_{I\in \binom{[n]}{k}}\in \R^{\binom{[n]}{k}}$
  is a \emph{positive tropical Pl\"ucker vector} 
 if for each three-term Pl\"ucker relation, it lies in
 the positive part of the associated tropical hypersurface:
	for  $1\leq a<b<c<d\leq n$
and $S \in \binom{[n]}{k-2}$  disjoint from $\{a,b,c,d\}$, 
either
		$$P_{Sac}+P_{Sbd} = P_{Sab}+P_{Scd} \leq P_{Sad}+P_{Sbc}
		\text{ or }
		P_{Sac}+P_{Sbd} = 
P_{Sad}+P_{Sbc} \leq 
P_{Sab}+P_{Scd}.$$
The \emph{positive Dressian} $\Dr^+_{k,n}$ is 
the set of 
positive
tropical Pl\"ucker vectors.  
\end{definition}

In general, the Dressian $\Dr_{k,n}$ contains the tropical Grassmannian $\Gr_{k,n}$ but is much larger 
\cite{Herrmann2008HowTD}.
However, the situation for their positive parts is different, see 
\cite{SW} and
\cite{PosConfig}.
\begin{theorem} [{\cite[Theorem 3.9]{SW} and \cite{PosConfig}}]
 The positive tropical Grassmannian $\Trop^+\Gr_{k,n}$ equals the positive Dressian
$\Dr^+_{k,n}$.  That is, all abstract positive tropical
linear spaces are realizable.
\end{theorem}

There are two natural fan structures on the Dressian -- the Pl\"ucker fan and the 
secondary fan -- which coincide \cite{Olarte}.  
The cones of 
$\Trop^+\Gr_{k,n}$ control the 
\emph{regular subdivisions} of $\Delta_{k,n}$ into positroid polytopes, with maximal cones giving
rise to positroid tilings.

Consider a point
$P=\{P_I\}_{I\in \binom{[n]}{k}}\in \R^{\binom{[n]}{k}}$,
which we also think of as 
a real-valued function $\{e_I\} \mapsto P_I$ on the vertices
of $\Delta_{k,n}$.  We define a polyhedral subdivision $\mathcal{D}_P$
of $\Delta_{k,n}$ as follows: consider the points
$(e_I, P_I)\in \Delta_{k,n} \times \R$ and take their convex hull.
Take the lower faces (those whose outwards normal vector have last component
negative) and project them  down to $\Delta_{k,n}$.  This 
gives us the \emph{regular subdivision} $\mathcal{D}_P$ of $\Delta_{k,n}$.
Note that $\mathcal{D}_P$ is a \emph{polytopal subdivision}
and that its facets (top-dimensional faces) comprise a 
positroid subdivision for the moment map
(cf.  \cref{def:tri}).

\begin{theorem} [{\cite[Theorem 9.12]{LPW}}, see also \cite{PosConfig}] \label{thm:trop}
			The point $P=\{P_I\}_{I \in \binom{[n]}{k}}$
                is a positive tropical Pl\"ucker vector if and only if 
		every face of 
		$\mathcal{D}_P$ 
			is a positroid polytope.
\end{theorem}


\begin{example}\label{ex:trop1}
Consider a positive tropical Pl\"ucker vector $P=\{P_I\}_{I\in {[4] \choose 2}} \in \R^{[4] \choose 2}$.
	If we have $P_{13}+P_{24} = P_{14}+P_{23} < P_{12}+P_{34},$
then $P$ induces the 
	subdivision 
		$\mathcal{D}_P$ 
	of $\Delta_{2,4}$ into the two square pyramids shown 
	in 
\cref{fig:pospolytopes}.
	If $P_{13}+P_{24} = P_{12}+P_{34} < P_{14}+P_{23},$
	we get the subdivision of $\Delta_{2,4}$ into two square pyramids
	separated by the square with vertices
	$e_{12}$, $e_{13}$, $e_{24}$, $e_{34}$, see 
        \cref{fig:T1}.
These are both positroid tilings of 
$\Delta_{2,4}.$  
\end{example}

Positroid tilings are particularly nice.
The following result refines \cite[Theorem 6.6]{SW}.
\begin{theorem}\cite[Theorem 6.6]{SW} \label{thm:finest}
	Let $P\in \Trop^+\Gr_{k,n}$ and consider the positroid
	subdivision $\mathcal{D}_P$.  The following statements
	are equivalent: 
	\begin{itemize}
		\item The facets of $\mathcal{D}_P$ comprise a positroid tiling.
		\item The facets of $\mathcal{D}_P$ comprise
			a finest positroid subdivision.
		\item Every face of $\mathcal{D}_P$ is the 
			matroid polytope of a series-parallel matroid.
   \item Every octahedron in $\mathcal{D}_P$ is subdivided.
	\end{itemize}
\end{theorem}

Combining \cref{thm:finest} with Speyer's $f$-vector theorem \cite{Ktheory}
gives the following. 

\begin{corollary}\label{cor:binomial}
	Let $\mathcal{D}_P$ be a positroid tiling of
	$\Delta_{k,n}$ as in \cref{thm:finest}.
	Then this polyhedral subdivision contains precisely
	$\frac{(n-c-1)!}{(k-c)!(n-k-c)!(c-1)!}$ interior faces of 
	dimension $n-c$.  In particular, $\mathcal{D}_P$ 
	consists of precisely ${n-2 \choose k-1}$ positroid tiles.
\end{corollary}


\subsection{Realizability of positively oriented matroids}

Positroid polytopes and the positive tropical Grassmannian have
applications to realizability questions.
One of the early goals of matroid theory was to find 
axioms that characterized the realizable matroids, i.e. those
that arise 
from full rank matrices as in \cref{def:realizable}.
However, this problem  (over a field
	of characteristic $0$) 
is now considered to be intractible
\cite{MNW2, Vamos}:  in V\'amos's words, 
 ``the missing axiom of matroid theory is lost forever".

There is a notion of \emph{oriented matroids}, in which 
bases have signs: an oriented matroid 
	is a matroid $([n], \B)$ together with a 
	\emph{chirotope},  a function
	$\chi:[n]^k \to \{0, +,-\}$ obeying certain axioms
which roughly encode the three-term Pl\"ucker relations \cite{OrientedMatroidBook}.
If $M$ is a full rank $k \times n$ real matrix, the function taking $(i_1, i_2, \ldots, i_k)$ to the sign of the minor using columns $(i_1, i_2, \dots, i_k)$ is a chirotope, and the realizable oriented matroids are precisely the chirotopes occurring in this way.
Thus, if $M$ represents a point of the positive
Grassmannian, then $M$ gives a chirotope $\chi$ with $\chi(i_1, i_2, \ldots, i_k) \in \{ 0, + \}$ for $1 \leq i_1 < i_2 < \cdots < i_k \leq n$.

A \emph{positively oriented matroid} is a chirotope $\chi$ such that
  $\chi(i_1, i_2, \ldots, i_k) \in \{ 0, + \}$ for $1 \leq i_1 < i_2 < \cdots < i_k \leq n$.
The following  was conjectured by da Silva in 1987 \cite{daS}.
\begin{theorem} \cite{ARW2, SW}
Every positively oriented matroid is realizable.  In other words,
every positively oriented matroid is a positroid.
\end{theorem} 
The first proof used the geometry of positroid polytopes \cite{ARW2}.
The second 
 used the fact that if $P \in \Trop \Gr_{k,n}$,
  every face of  the subdivision $\mathcal{D}_P$
 corresponds to a realizable matroid \cite{SW}.

\section{The amplituhedron and the sign stratification}
\label{sec:stratification}

Building on \cite{abcgpt},
Arkani-Hamed and Trnka 
\cite{arkani-hamed_trnka} 
introduced 
the \emph{(tree) amplituhedron}, which
is the image of the positive Grassmannian under a positive linear map.
Let $\Mat_{n,p}^{>0}$ denote the set of $n\times p$ matrices whose maximal minors
are positive.
Throughout this section 
we will make the convention that $k,n,m$ are positive integers such that 
 $k+m \leq n$.

\begin{defn}\label{defn_amplituhedron}
Let $Z\in \Mat_{n,k+m}^{>0}$. 
	The \emph{amplituhedron map} 
$\tilde{Z}:\Gr_{k,n}^{\ge 0} \to \Gr_{k,k+m}$
      is
	defined by
	$\tilde{Z}(C):=CZ$,
	where 
 $C$ is a $k \times n$ matrix representing an element of
	$\Gr_{k,n}^{\ge 0}$,  and  $CZ$ is a $k \times (k+m)$ matrix representing an 
         element of $\Gr_{k,k+m}$.
	The \emph{amplituhedron} $\mathcal{A}_{n,k,m}(Z) \subset \Gr_{k,k+m}$ is the image
$\tilde{Z}(\Gr_{k,n}^{\ge 0})$.
\end{defn}

The condition $Z \in \Mat_{n,k+m}^{>0}$  ensures that 
$\rank(CZ)=k$
 and hence $\tilde{Z}$
is well defined
\cite{arkani-hamed_trnka}.
This condition  can be relaxed: 
the map $\tilde{Z}$ is well-defined
if and only if $\var(v) \geq k$ for all nonzero $v\in \ker(Z)$
 \cite[Theorem 4.2]{karp}.

If $k+m=n$,  $\mathcal{A}_{n,k,m}(Z)$ is isomorphic to
$\Gr_{k,k+m}^{\geq 0}$.
The amplituhedron $\mathcal{A}_{n,k,m}(Z)$ has full dimension $km$
inside $\Gr_{k,k+m}$,  but it 
does not 
lie inside $\Gr_{k,k+m}^{\geq 0}$ in general.

\begin{example}\label{ex:polytope}
If $k=1$ and $m=2$, $\mathcal{A}_{n,1,2}(Z)$ is a polygon in 
projective space $\PP^2$.  
To see this, let $Z_1,\dots, Z_n$ denote
the rows of $Z\in \Mat_{n,3}^{>0}$; we can think of each $Z_i$ as a point in $\PP^{2}$.
Since $Z\in 
\Mat_{n,3}^{>0}$, the points $Z_1,\dots,Z_n$ are arranged in convex position
like the vertices of a polygon, see 
	\cref{fig:T2}.
Elements $C\in \Gr_{1,n}^{\geq 0}$ are just non-zero vectors
with 
	coordinates in $\R_{\geq 0}$.  If $C = e_i$ then $CZ = Z_i$, so all points 
$Z_i$ lie in $\mathcal{A}_{n,1,2}$.  As $C$ 
varies over 
$\Gr_{1,n}^{\geq 0}$, the image $CZ$ varies over all convex
	combinations of the $Z_i$'s, and hence $\mathcal{A}_{n,1,2}(Z)$ is an $n$-gon.
For $k=1$ and $m$ general, it follows from \cite{Sturmfels} that 
$\mathcal{A}_{n,1,m}(Z)$ is a {\itshape cyclic polytope} in  $\PP^m$.  
\end{example}

Physicists are most interested in the amplituhedron
$\A_{n,k,4}(Z)$.  Since 
$\AA(Z) \subset \Gr_{k,k+m}$, when $m$ is small it is sometimes more convenient
to take orthogonal complements and work with $\Gr_{m,k+m}$ instead of $\Gr_{k,k+m}$.  This motivates the following definition of the 
\emph{B-amplituhedron}, which is homeomorphic to the 
``A-amplituhedron.''

\begin{defn}[{\cite[Definition 3.8]{karpwilliams}}]
	\label{def:B}
       Choose $Z \in \Mat_{n,k+m}^{>0}$, 
	and let 
	$W\in \Gr_{k+m,n}^{>0}$ 
	be the column span of $Z$.
	We define the \emph{B-amplituhedron} to be
	$$\B_{n,k,m}(W):=\{V^{\perp} \cap W \ \vert \ V\in \Grk\} \subset 
	\Gr_{m}(W),$$ 
	where $\Gr_m(W)$ denotes the Grassmannian of $m$-planes 
	in $W$.
\end{defn}

The idea behind the identification 
$\B_{n,k,m}(W) \cong \A_{n,k,m}(Z)$
is that we obtain 
$\B_{n,k,m}(W) \subset \Gr_m(W) \subset \Gr_{m,n}$ from 
$\A_{n,k,m}(Z) \subset \Gr_{k,k+m}$ by taking orthogonal complements
in $\R^{k+m}$, then applying an isomorphism 
between $\R^{k+m}$ and $W$
so that our subspaces lie in $W$, not $\R^{k+m}$.

\begin{prop} [{\cite[Lemma 3.10 and Proposition 3.12]{karpwilliams}}]
\label{prop:AB}
     Choose $Z \in \Mat_{n,k+m}^{>0}$, 
and let 
	$W\in \Gr_{k+m,n}^{>0}$ 
	 be the column span of $Z$.
	We 
	 define a map $f_Z: \Gr_m(W) \to \Gr_{k,k+m}$ by 
	$$f_Z(X):= Z(X^{\perp}) = \{Z(x): x\in X^{\perp}\}.$$
	Then 
$f_Z$ restricts to a homeomorphism from 
	$\B_{n,k,m}(W)$ onto $\AA(Z)$, sending
	$V^{\perp} \cap W$ to $\tilde{Z}(V)$ for all 
	$V\in \Grk$.    
\end{prop}

We next discuss coordinates on $\AA(Z)\subset \Gr_{k,k+m}$ and 
$\B_{n,k,m}(W) \subset \Gr_{m,n}$.

\begin{definition}\label{def:twistor}
       Choose
        $Z \in \Mat_{n,k+m}^{>0}$ 
	with rows  $Z_1,\dots, Z_n \in \R^{k+m}$.
	Given a matrix $Y$ 
	with rows $y_1,\dots,y_k$
	representing an element of $\Gr_{k,k+m}$, 
	and $\{i_1,\dots, i_m\} \subset [n]$,
	we define the \emph{twistor coordinate}, denoted 
	$$\langle Y Z_{i_1} Z_{i_2} \cdots Z_{i_m} \rangle \text{ or }
        \langle y_1, \dots,y_k, Z_{i_1}, \dots, Z_{i_m} \rangle,$$
        to be  the determinant of the
	matrix with rows 
        $y_1, \dots,y_k, Z_{i_1}, \dots, Z_{i_m}$.
        
	If $I=\{i_1<\dots <i_m\}$, we also 
	use  $\lrangle{Y Z_I}$ to denote 
        $\langle Y Z_{i_1} Z_{i_2} \cdots Z_{i_m} \rangle$.
\end{definition}

\cref{lem:coordinates} shows that the 
twistor coordinates of the amplituhedron
$\AA(Z) \subset \Gr_{k,k+m}$
are equal to the Pl\"ucker coordinates of the B-amplituhedron 
$\B_{n,k,m}(W) \subset \Gr_{m,n}$.
\begin{lemma} [{\cite[(3.11)]{karpwilliams}}]
	\label{lem:coordinates}
	If we let $Y:=f_Z(X)$ in \cref{prop:AB},
	we have
	\begin{equation}\label{eq:identify}
	p_{I}(X) = \lrangle{Y Z_{I}}
		\qquad \text{ for all } \quad
		I \in {[n] \choose m}. 
	\end{equation}
\end{lemma}

	By \cref{def:B}
and \cref{gantmakher_krein}, 
if $X \in \B_{n,k,m}(W)$ and $w\in X\setminus \{0\}$ then 
$$k \leq \overline{\var}(w) \leq k+m-1.$$
When $m=1$ this leads to the following 
sign variation description
of the amplituhedron.

	\begin{theorem}
		[{\cite[Corollary 3.19]{karpwilliams}
		and \cref{lem:coordinates}}] 
	 We have 
	\begin{align}
		\label{b:flip}
		\B_{n,k,1}(W) &= \{w \in \PP(W) \ \vert \ \overline{\var}(w)=k\} \subset
	\PP(W), \text { or equivalently,} \\
		\label{a:flip}
		\A_{n,k,1}(Z) &= 
	\{Y \in \Gr_{k,k+1} \ \vert \ 
	\overline{\var}(\lrangle{YZ_1}, \lrangle{YZ_2},\dots,
	\lrangle{YZ_n}) = k \}.
	\end{align}
Moreover, $\mathcal{A}_{n,k,1}(Z) \cong \B_{n,k,1}(W)$ 
can be identified with the
complex of bounded faces of a cyclic hyperplane arrangement.
\end{theorem}
To make the latter identification, we choose a 
basis $w^{(0)}, w^{(1)}, \dots, w^{(k)}$ for $W\subset \R^n$
such that $\spn(w^{(1)},\dots,w^{(k)}) \in \Gr_{k,n}^{>0}$
	and $w^{(0)}$ is \emph{positively oriented}
as in \cite[Definition 6.10]{karpwilliams}.
	Then we let $\mathcal{H}^W$ be the hyperplane 
	arrangement in $\R^k$ with hyperplanes
	$$H_i:=\{x\in \R^k \ \vert \ 
	w_i^{(1)} x_1 + \dots + w^{(k)}_i x_k + w^{(0)}_i=0\}
	\quad \text { for } \quad i\in [n].$$
	By \cite[Theorem 6.16]{karpwilliams}, the map
	$$\Psi_{\mathcal{H}^W}: \R^k \to \PP(W),
	\quad \quad
	x \mapsto \lrangle{x_1 w^{(1)} + \dots + x_k w^{(k)} + w^{(0)}}$$
	is a homeomorphism from the bounded complex
	$B(\mathcal{H}^W)$ of 
	$\mathcal{H}^W$ to $\mathcal{B}_{n,k,1}(W)$.
	Moreover, if we partition the elements
	$w=(w_1,\dots,w_n)\in \mathcal{B}_{n,k,1}(W)$
	based on the signs of the $w_i$,  we obtain
	the bounded
	faces of 
	$B(\mathcal{H}^W)$.
	See \cref{fig:hyperplane}.

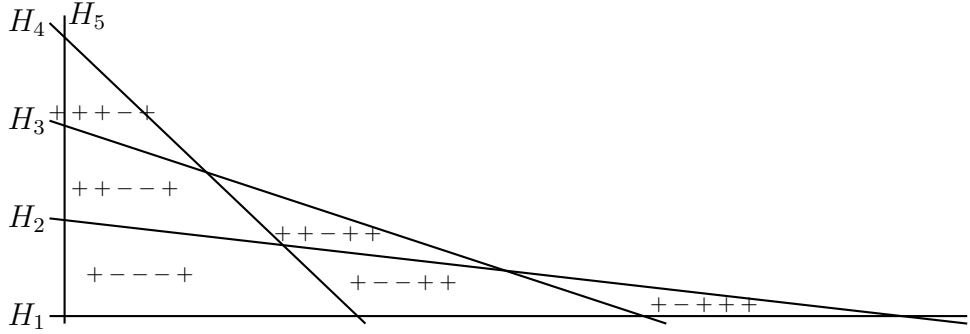
\begin{figure}[h]
\begin{tikzpicture}
	\draw[black, thick] (-.2,0)--(12,0);
	\draw[black, thick] (-.2,3.9)--(4,-.1);
	\draw[black, thick] (-.2,2.6)--(8,-.1);
	\draw[black, thick] (-.2,1.3)--(12,-.1);
	\draw[black, thick] (0,-.1)--(0,4);
	\draw (-.5, 0) node {$H_1$};
	\draw (-.5, 1.3) node {$H_2$};
	\draw (-.5, 2.6) node {$H_3$};
	\draw (-.5, 3.9) node {$H_4$};
	\draw (.3, 4) node {$H_5$};
	\draw (.5,2.7) node {\tiny $+++-+$};
	\draw (.8,1.7) node {\tiny $++--+$};
	\draw (1,.55) node {\tiny $+---+$};
	\draw (4.5,.45) node {\tiny $+--++$};
	\draw (3.5,1.1) node {\tiny $++-++$};
	\draw (8.5,.15) node {\tiny $+-+++$};
\end{tikzpicture}
	\caption{The hyperplane arrangement ${\mathcal{H}^W}$ from $\mathcal{B}_{5,2,1}(W)\cong B({\mathcal{H}^W})$, where 
	$w^{(0)} = (0,-9,-6, -3, 0)$, $w^{(1)} = (0,1,1,1,1)$, and $w^{(2)} = (1,9,3,1,0)$. 
	Its bounded faces are labeled by sign vectors. Here we have labeled only the maximal faces.}
	\label{fig:hyperplane}
\end{figure}


Before turning to the case $m=2$, we need to 
introduce some notation.

\begin{remark}
\label{rem:twisted}
Let $Z\in \Mat^{\geq 0}_{n,p}$ with $n \geq p$ have rows $Z_1,Z_2, \dots, Z_n$, and let $\hat{Z}_i$ denote
$(-1)^{p-1} Z_i$.
Then the matrix with rows
 $Z_2,\dots, Z_n, \hat{Z}_1$ also lies in $\Mat^{>0}_{n,p}$.  Thus matrices with maximal minors
	nonnegative (and elements of $\Grk$)
	exhibit 
a \emph{twisted cyclic symmetry}.
\end{remark}

The following sign variation description
of the amplituhedron was conjectured by
Arkani-Hamed--Thomas--Trnka
\cite{ATT}. 

\begin{theorem}[{\cite[Theorem 5.1]{PSBW}}]
	\label{thm:main1}
Fix $k<n$  and $Z\in \Mat^{>0}_{n,k+2}$.  
        Let
        \begin{align*}
                \mathcal{F}^\circ_{n,k,2}(Z):=  \{Y\in Gr_{k,k+2} \ &\vert \
        \langle Y Z_i Z_{i+1}\rangle>0 \text{ for }1 \leq i \leq n-1,
        \text{ and }   \langle Y Z_n \hat{Z}_1 \rangle >0,\\
                &\text{ and }
                \var({\langle Y Z_1 Z_2\rangle,
        \langle Y Z_1 Z_3\rangle,  \dots
                \langle Y Z_1 Z_n\rangle})=k.\}
        \end{align*}
        Then $\mathcal{A}_{n,k,2}(Z) = \overline{\mathcal{F}^\circ_{n,k,2}(Z)}.$
\end{theorem}


To generalize \eqref{a:flip}
	and \cref{thm:main1} for $m>2$, 
we first observe the following \cite{ATT}.
\begin{lemma}
\label{lem:boundary}
Let $Z\in \Mat_{n,k+m}^{>0}$ and $Y=CZ$ 
	for $C\in \Gr_{k,n}^{>0}$. 
Let $r = \lfloor{\frac{m}{2} \rfloor}.$
	\begin{align*}
		&\text{If $m$ is even, } \lrangle{Y Z_I} > 0 \text{ for any set }
	I:=\{i_1 < i_1+1 < i_2<i_2+1< \dots < i_{r} < i_{r}+1\} \in {[n] \choose m}\\
		&\quad \text{ and } \lrangle{Y Z_I Z_n \hat{Z}_1} >0 \text{ for any set }
		I:=\{i_1 < i_1+1 < \dots < i_{r-1} < i_{r-1}+1\}\in {[2,n-1] \choose m-2}.\\
		&\text{If $m$ is odd, }
		(-1)^k \lrangle{Y Z_I} > 0 \text{ for any set }
	I:=\{1=i_0 < i_1 < i_1+1 < \dots < i_{r} < i_{r}+1\} \in {[n] \choose m} \\
		&\quad \text{ and } \lrangle{Y Z_I}>0 \text{ for any set  }
		I:=\{i_1 < i_1+1 < \dots < i_{r} < i_{r}+1<i_{r+1}=n \} \in {[n] \choose m}.
	\end{align*}
\end{lemma}
\begin{proof}
	The lemma follows from the fact that $Z\in \Mat_{n,k+m}^{>0}$ and
	 \begin{equation}\label{eq:expand}
                \langle CZ, Z_{i_1}, \dots, Z_{i_m} \rangle =
                \sum_{J=\{j_1<\dots<j_k\} \in {[n] \choose k}} p_J(C) \langle Z_{j_1}, \dots, Z_{j_k}, Z_{i_1},\dots, Z_{i_m}\rangle,
        \end{equation}
	 by considering how many swaps are necessary
	to put $\{j_1,\dots, j_k, i_1,\dots, i_m\}$ in order.
\end{proof}

\begin{proposition}
		[{\cite[Section 5.4]{ATT}
and \cite[Corollary 3.21]{karpwilliams}}]
	\label{conj:signflip}
Fix $Z\in \Mat^{>0}_{n,k+m}$, and define
  \begin{align*}
	  \mathcal{F}^\circ_{n,k,m}(Z)&:=  
	 \{Y\in Gr_{k,k+m} \ \vert \
	  \text{ the conclusion of \cref{lem:boundary} holds, and }  \\        
		&\var({\langle Y Z_1 \dots Z_{m-1} Z_m\rangle,
	  \langle Y Z_1 \dots Z_{m-1} Z_{m+1} \rangle,  \dots , 
		\langle Y Z_1 \dots Z_{m-1} Z_n\rangle})=k.\}
        \end{align*}
 Then 
$\mathcal{A}_{n,k,m}(Z) \subseteq \overline{\mathcal{F}^\circ_{n,k,m}(Z)}.$
\end{proposition}
It is expected that the inequality above is an equality for 
$m=4$, see \cite[(5.14)]{ATT}.


Given the evident importance of signs of coordinates, and 
taking inspiration from the 
$m=1$ 
example,
we define the \emph{sign stratification}
for 
the amplituhedron; this stratification is closely 
related to the \emph{oriented matroid stratification} of the Grassmannian.

\begin{definition}\label{def:chamber}
Let $\sigma= (\sigma_{I})\in \{0, +,-\}^{n \choose m}$ be a nonzero sign vector with coordinates indexed by subsets
	$I \in {[n] \choose m}$.  We consider $\sigma$
         \emph{modulo multiplication by
        $\pm 1$}  (since Pl\"ucker and twistor coordinates are
	coordinates in projective space).
        Set 
	\begin{align*}
\mathcal{A}^{\sigma}_{n,k,m}(Z)&:=\{Y\in \mathcal{A}_{n,k,m}(Z) \ \vert \
	\sign \lrangle{Y Z_{I}} = \sigma_{I}
		\text{ for all }I\} \quad \text{ and }\\
	\mathcal{B}^{\sigma}_{n,k,m}(W)&:=\{X\in \mathcal{B}_{n,k,m}(W) \ \vert \
		\sign (p_{I}(X))  = \sigma_{I}
		\text{ for all }I\}.
	\end{align*}
  We call
	$\mathcal{A}^{\sigma}_{n,k,m}(Z)$ (respectively,
	$\mathcal{B}^{\sigma}_{n,k,m}(W)$) an \emph{(amplituhedron) sign stratum.}
	Clearly \begin{equation*}
		\mathcal{A}_{n,k,m}(Z) = \sqcup_{\sigma} \mathcal{A}^{\sigma}_{n,k,m}(Z)
		\quad \text{ and } \quad 
	\mathcal{B}_{n,k,m}(W) = \sqcup_{\sigma} \mathcal{B}^{\sigma}_{n,k,m}(W).
	\end{equation*}
        If $\sigma\in \{+,-\}^{n \choose m}$,
        we call
		$\mathcal{A}^{\sigma}_{n,k,m}(Z)$ and $\mathcal{B}^{\sigma}_{n,k,m}(W)$  open
        \emph{(amplituhedron) chambers}.
\end{definition}

For arbitrary $\sigma$, 
$\mathcal{A}^{\sigma}_{n,k,m}(Z)$ (or $\mathcal{B}^{\sigma}_{n,k,m}(W)$)
may be empty.  We 
call $\mathcal{A}^{\sigma}_{n,k,m}$ 
 \emph{realizable} if there is 
some $Z$ for which 
$\mathcal{A}^{\sigma}_{n,k,m}(Z)$ 
is nonempty. 
It is an open problem to classify
the realizable chambers/strata of the amplituhedron for $m>2$.
When $m=1$, 
the realizable strata 
are precisely the $\mathcal{A}^{\sigma}_{n,k,1}(Z)$ with
$\overline{\var}(\sigma)=k$ (cf. 
		\eqref{a:flip} and 
\cite{karpwilliams}). 
When $m=2$, we will 
show in 
 \cref{sec:T} that 
the realizable chambers
 of $\mathcal{A}_{n,k,2}$ 
 are counted by 
the \emph{Eulerian numbers}. 
This is  related to the fact that
the volume of  $\Delta_{k+1,n}$ is the Eulerian number.

\section{The amplituhedron map, positroid tilings, and plane partitions}
\label{sec:amp}

In this section we  begin our discussion of positroid tilings 
of the amplituhedron, cf. \cref{def:tri} with $\phi = \tilde{Z}$.
These have also beeen called \emph{(positroid) triangulations}.

\begin{definition}
        Choose $Z\in \Mat_{n,k+m}^{>0}$.
Given a positroid cell $S_{\pi}$ of
$\Gr_{k,n}^{\ge 0}$, we let
$\gto{\pi} := \tilde{Z}(S_{\pi})$ and
$\gt{\pi} := \overline{
        \tilde{Z}(S_{\pi})} = \tilde{Z}(\overline{S_{\pi}})$,
        and we refer to
	$\gt{\pi}$ 
	and 
$\gto{\pi}$ 
	as 
\emph{Grasstopes} and 
	\emph{open Grasstopes}, 
	respectively.
	As in \cref{def:tri}, we call
        $\gt{\pi}$  and $\gto{\pi}$
	a \emph{(positroid) tile}  and an \emph{open (positroid) tile}
        for $\mathcal{A}_{n,k,m}(Z)$
        if $\dim(S_\pi) =km$ and
         $\tilde{Z}$ is injective on $S_\pi$.
	 Since positroid cells are also indexed by plabic graphs,
	 we will also use $Z_G$ and $Z^{\circ}_G$ to denote the Grasstopes
	 associated to the positroid cell $S_G$.
\end{definition}
 
 Images of positroid cells under the map $\tilde{Z}$
have been studied since \cite{arkani-hamed_trnka}, where
the authors
conjectured that the images
of certain \emph{BCFW} collections of $4k$-dimensional  cells in
$\Gr_{k,n}^{\geq 0}$ give a ``triangulation'' (positroid tiling) of the amplituhedron
$\mathcal{A}_{n,k,4}(Z)$.   This conjecture was recently proved in 
a beautiful paper of 
Even-Zohar--Lakrec--Tessler \cite{ELT}, using the BCFW collection of cells
studied in \cite{karp:2017ouj}.


The positroid tiles for $\mathcal{A}_{n,k,m}(Z)$ 
were classified for $m=1$ in \cite[Theorem 8.10]{karpwilliams}.
For $m=2$ and $k=1$ 
(cf. \cref{ex:polytope}),
 the amplituhedron $\mathcal{A}_{n,1,2}(Z)$ 
is a convex $n$-gon in $\PP^2$. The positroid tiles are exactly the
  triangles on vertices $Z_1,\dots,Z_n$ of
 the polygon, and positroid tilings of $\mathcal{A}_{n,1,2}(Z)$
 are just ordinary triangulations of a polygon.
More generally, the positroid tiles have been classified for $m=2$ 
 in \cite[Theorem 4.25]{PSBW}, as we now describe.

\subsection{Classification of positroid tiles for the amplituhedron map when $m=2$}

When $m=2$,  positroid tiles 
are in bijection with \emph{bicolored subdivisions} of an $n$-gon. 

\begin{definition} 
\label{def:param}
Let $\mathbf{P}_n$ be a convex $n$-gon with boundary vertices labeled 
from $1$ to $n$ in clockwise order.
A \emph{bicolored triangulation} $\mathcal{T}$ of $\mathbf{P}_n$
is a triangulation whose edges connect vertices of $\mathbf{P}_n$
and whose triangles are colored black or white.
If $\mathcal{T}$ 
has exactly $k$ black triangles, we say it has
\emph{type $(k,n)$}.
Two bicolored triangulations
$\T$ and $\T'$  are
\emph{equivalent}
if the union of the black triangles is the same for both of them.
 If we erase the diagonals separating 
pairs of like-colored triangles in $\T$
	sharing an edge, we obtain
	a \emph{bicolored subdivision} 
        $\overline{\T}$
	of $\mathbf{P}_n$ into black and  white 
        polygons,   which represents the equivalence class of $\T$.


Given 
$\T$ as above, we build a labeled bipartite graph $\hatG(\T)$ by
placing black boundary vertices
        at the  vertices of the $n$-gon,
        and placing a trivalent white vertex in the middle of
        each black triangle, connecting it to the three vertices of the
        triangle.
See \cref{fig:Kasteleyn}.  
We can think of $\hatG(\T)$ as a \emph{plabic graph} if
we enclose it
in a slightly larger disk and
add $n$ edges connecting each black vertex  to the boundary of the disk.
	See \cref{fig:Kasteleyn}.  
\end{definition}

Note that if $\T$ and $\T'$ are equivalent, $\hatG(\T)$ and $\hatG(\T')$
are move-equivalent.
Bicolored triangulations and subdivisions 
are special cases of \emph{plabic tilings} \cite{OPS}.

\begin{figure}[h]
	\centering
\resizebox{1.4in}{!}{
\begin{tikzpicture}
	\node[draw, minimum size=3.3cm, regular polygon, regular polygon sides=9] (a) {};
	 \filldraw[fill=gray!30!white](a.corner 1)--(a.corner 2)--(a.corner 3)--(a.corner 4)--(a.corner 1);
	 \filldraw[fill=gray!30!white](a.corner 4)--(a.corner 6)--(a.corner 7)--(a.corner 8)--(a.corner 9)--(a.corner 4);
	\foreach \x in {1,2,...,9}
	\fill (a.corner \x) circle[radius=2pt];
	\draw[black,thin](a.corner 1)--(a.corner 4);
	\draw[black,thin](a.corner 9)--(a.corner 4);
	\draw[black,thin](a.corner 6)--(a.corner 4);
	\draw[black,thin](a.corner 2)--(a.corner 4);
	\draw[black,thin](a.corner 7)--(a.corner 4);
	\draw[black,thin](a.corner 8)--(a.corner 4);
	 \node[shift=(a.corner 1), anchor=south] {$1$};
	 \node[shift=(a.corner 2),anchor=south] {$9$};
	 \node[shift=(a.corner 3),anchor=east] {$8$};
	 \node[shift=(a.corner 4),anchor=east] {$7$};
	 \node[shift=(a.corner 5),anchor=north] {$6$};
	 \node[shift=(a.corner 6),anchor=north] {$5$};
	 \node[shift=(a.corner 7),anchor=west] {$4$};
	 \node[shift=(a.corner 8),anchor=west] {$3$};
	 \node[shift=(a.corner 9),anchor=west] {$2$};
\end{tikzpicture}
	}
\resizebox{1.4in}{!}{
\begin{tikzpicture}
	\node[draw, minimum size=3.3cm, regular polygon, regular polygon sides=9] (a) {};
	 \filldraw[fill=gray!30!white](a.corner 1)--(a.corner 2)--(a.corner 3)--(a.corner 4)--(a.corner 1);
	 \filldraw[fill=gray!30!white](a.corner 4)--(a.corner 6)--(a.corner 7)--(a.corner 8)--(a.corner 9)--(a.corner 4);
	\foreach \x in {1,2,...,9}
	\fill (a.corner \x) circle[radius=2pt];
	\draw[black,thin](a.corner 1)--(a.corner 4);
	\draw[black,thin](a.corner 9)--(a.corner 4);
	\draw[black,thin](a.corner 6)--(a.corner 4);
	 \node[shift=(a.corner 1), anchor=south] {$1$};
	 \node[shift=(a.corner 2),anchor=south] {$9$};
	 \node[shift=(a.corner 3),anchor=east] {$8$};
	 \node[shift=(a.corner 4),anchor=east] {$7$};
	 \node[shift=(a.corner 5),anchor=north] {$6$};
	 \node[shift=(a.corner 6),anchor=north] {$5$};
	 \node[shift=(a.corner 7),anchor=west] {$4$};
	 \node[shift=(a.corner 8),anchor=west] {$3$};
	 \node[shift=(a.corner 9),anchor=west] {$2$};
\end{tikzpicture}
}
\resizebox{1.4in}{!}{
\begin{tikzpicture}
	\node[draw, minimum size=3.3cm, regular polygon, regular polygon sides=9] (a) {};
	 \filldraw[fill=gray!30!white](a.corner 1)--(a.corner 2)--(a.corner 3)--(a.corner 4)--(a.corner 1);
	 \filldraw[fill=gray!30!white](a.corner 4)--(a.corner 6)--(a.corner 7)--(a.corner 8)--(a.corner 9)--(a.corner 4);
	\foreach \x in {1,2,...,9}
	\fill (a.corner \x) circle[radius=2pt];
	\draw[black,thin](a.corner 1)--(a.corner 4);
	\draw[black,thin](a.corner 9)--(a.corner 4);
	\draw[black,thin](a.corner 6)--(a.corner 4);
	\draw[black,thin](a.corner 2)--(a.corner 4);
	\draw[black,thin](a.corner 7)--(a.corner 4);
	\draw[black,thin](a.corner 8)--(a.corner 4);
	 \node[shift=(a.corner 1), anchor=south] {$1$};
	 \node[shift=(a.corner 2),anchor=south] {$9$};
	 \node[shift=(a.corner 3),anchor=east] {$8$};
	 \node[shift=(a.corner 4),anchor=east] {$7$};
	 \node[shift=(a.corner 5),anchor=north] {$6$};
	 \node[shift=(a.corner 6),anchor=north] {$5$};
	 \node[shift=(a.corner 7),anchor=west] {$4$};
	 \node[shift=(a.corner 8),anchor=west] {$3$};
	 \node[shift=(a.corner 9),anchor=west] {$2$};
        \draw[blue,thick](-.8,.8)--(a.corner 1);
        \draw[blue,thick](-.8,.8)--(a.corner 2);
        \draw[blue,thick](-.8,.8)--(a.corner 4);
        \draw[blue,thick](-1.4,.3)--(a.corner 2);
        \draw[blue,thick](-1.4,.3)--(a.corner 3);
        \draw[blue,thick](-1.4,.3)--(a.corner 4);
        \draw[blue,thick](0.5,.3)--(a.corner 4);
        \draw[blue,thick](0.5,.3)--(a.corner 9);
        \draw[blue,thick](0.5,.3)--(a.corner 8);
        \draw[blue,thick](0.3,-1.1)--(a.corner 4);
        \draw[blue,thick](0.3,-1.1)--(a.corner 6);
        \draw[blue,thick](0.3,-1.1)--(a.corner 7);
        \draw[blue,thick](0.8,-.4)--(a.corner 4);
        \draw[blue,thick](0.8,-.4)--(a.corner 7);
        \draw[blue,thick](0.8,-.4)--(a.corner 8);
	\filldraw[color=blue,fill=white] (-.8,.8) circle (2pt);
	\filldraw[color=blue,fill=white] (-1.4,.3) circle (2pt);
	\filldraw[color=blue,fill=white] (0.5,.3) circle (2pt);
	\filldraw[color=blue,fill=white] (0.3,-1.1) circle (2pt);
	\filldraw[color=blue,fill=white] (0.8,-.4) circle (2pt);
\end{tikzpicture}
}
	\caption{A bicolored triangulation 
	$\T$; the corresponding 
	bicolored subdivision  $\overline{\T}$;  the 
	graph $\hatG(\T)$.}
        \label{fig:Kasteleyn}
\end{figure}
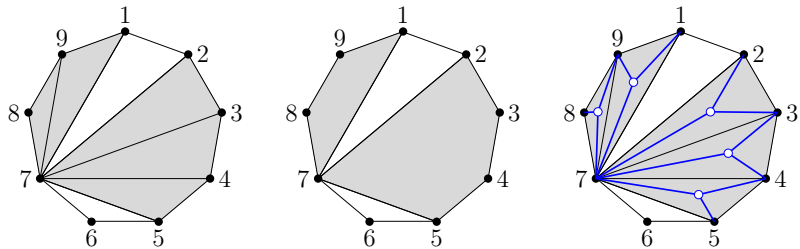


The following statement refines a  conjecture from 
 \cite{Lukowski:2019sxw}. 

\begin{theorem}
	[{\cite[Theorem 4.25]{PSBW}}]
	\label{thm:allGTs} 
	Fix $k<n$ and $Z \in \Mat^{>0}_{n, k+2}$. Then $\tilde{Z}$ is injective on the $2k$-dimensional cell $S_\M$ if and only if  $S_\M = S_{\hatG(\T)}$ for some bicolored triangulation $\T$ of type $(k,n)$.
That is, the positroid tiles for $\A_{n, k, 2}$ are exactly
        the Grasstopes $\gt{\hatG(\T)}$.
\end{theorem}

It is an open problem to classify positroid tiles of 
$\AA(Z)$ for $m>2$.

\subsection{Numerology of positroid tilings of the amplituhedron}

When $m=4$, each (conjectural) BCFW  tiling
of $\A_{n,k,4}(Z)$ has cardinality 
equal to the \emph{Narayana number}
$N_{n-3,k+1}=\frac{1}{n-3} \binom{n-3}{k+1} \binom{n-3}{k}.$
What should be the cardinality 
for $m\neq 4$?
The following table gives  data about  special cases studied
thus far.

\begin{center}
\setlength{\tabcolsep}{5pt}
\begin{tabular}{|c|c|c|}
\hline 
	{\bfseries special case} & {\bfseries cardinality of tiling of 
	$\mathcal{A}_{n,k,m}(Z)$} 
	& {\bfseries explanation} \\ \hline
	$m=0$\text{ or }$k=0$ \text{ or }$k+m=n$ & $1$\rule{0pt}{13pt}\rule[-7pt]{0pt}{0pt} & 
	$\mathcal{A}$ is a point or 
	$\mathcal{A}\cong\Gr_{k,n}^{\ge 0}$  
	\\ \hline
	$m=1$ & $\displaystyle\binom{n-1}{k}$\rule{0pt}{21pt}\rule[-14pt]{0pt}{0pt} & \cite{karpwilliams} \\ \hline
$m=2$ & $\displaystyle\binom{n-2}{k}$\rule{0pt}{21pt}\rule[-14pt]{0pt}{0pt} & \cite{ATT, BaoHe, PSBW} \\ \hline
$m=4$ & $\displaystyle\frac{1}{n-3}\binom{n-3}{k+1}\binom{n-3}{k}$\rule{0pt}{21pt}\rule[-14pt]{0pt}{0pt} &  \cite{arkani-hamed_trnka, ELT} \\ \hline
$k=1$, $m$ even & $\displaystyle  \binom{n-1- \frac{m}{2} }{ \frac{m}{2}}$\rule{0pt}{21pt}\rule[-14pt]{0pt}{0pt} & $\mathcal{A}\cong \text{cyclic polytope }C(n,m)$
	\\ \hline
\end{tabular}
\end{center}

As we will see later, the appearance of the number ${n-2 \choose k}$ in the $m=2$ row of the table 
is related to the appearance of the number ${n-2 \choose k}$ in 
\cref{cor:binomial}.

The special cases in the table led us to make
the following intriguing
conjecture.

\begin{conj}[{\cite[Conjecture 8.1]{karp:2017ouj}}]\label{conj:plane}
If $m$ is even, the cardinality of a  positroid tiling of the amplituhedron
$\mathcal{A}_{n,k,m}(Z)$ is 
$M(k, n-k-m,  \frac{m}{2} )$, where
$$M(a,b,c) := \prod_{i=1}^a\prod_{j=1}^b\prod_{k=1}^c\frac{i+j+k-1}{i+j+k-2}
$$
	is the number of \emph{plane partitions}  contained in an $a \times b \times c$ box.
\end{conj}
For odd $m$, we believe that 
the \emph{maximum} cardinality achieved by 
a positroid tiling of $\mathcal{A}_{n,k,m}(Z)$ is 
 $M(k, n-k-m, \lceil \frac{m}{2} \rceil)$.
This is consistent with the results of \cite{karpwilliams} for $m=1$, and the
fact that for odd $m$,
the number of top-dimensional simplices in a triangulation of the cyclic polytope $C(n,m)$ 
can lie anywhere between $\binom{n-1-\frac{m+1}{2}}{\frac{m-1}{2}}$ and 
	$\binom{n-\frac{m+1}{2}}{\frac{m+1}{2}}$ \cite[Corollary 1.2(ii)]{rambau_97}.  

\begin{remark}
Clearly $M(a,b,c)$ is symmetric in $a, b$ and $c$.  
\cref{conj:plane} thus suggests that for even $m$, there is a symmetry of 
(the positroid tilings of)	the amplituhedron 
$\mathcal{A}_{n,k,m}$ that allows one to exchange
	$k$, $n-k-m$, and $\frac{m}{2}$.  The 
	symmetry
	between $k$ and $n-k-m$ is called \emph{parity duality}
	and was subsequently verified in \cite{Galashin:2018fri}.
\end{remark}

\begin{remark}
	Besides counting plane partitions, 
	the numbers $M(a,b,c)$ have multiple combinatorial interpretations.
	For example, they also count collections of 
	$c$ noncrossing lattice paths inside an $a\times b$ rectangle; 
	rhombic tilings of a hexagon with side lengths $a, b, c, a, b,c$; and 
	perfect matchings of a certain honeycomb lattice.
	See \cite[Section 8.1]{karp:2017ouj}.
\end{remark}

\section{T-duality and positroid tilings of 
$\Delta_{k+1,n}$ and $\mathcal{A}_{n,k,2}$ }
\label{sec:T}

In this section we will fix $1 \leq k \leq n-2$, and explore a  mysterious
duality between
the hypersimplex $\Delta_{k+1,n}$ -- an $(n-1)$-dimensional polytope in $\R^n$ -- and 
the amplituhedron $\mathcal{A}_{n,k,2}(Z)$ -- a $2k$-dimensional (non-polytopal) subset
of $\Gr_{k,k+2}$ \cite{LPW, PSBW}.  This duality was first discovered in 
\cite{LPW},
after we observed that for small $k$ and $n$,
the $f$-vector of the 
positive tropical Grassmannian $\Trop^+ \Gr_{k+1,n}$ \cite{troppos} 
agrees with the numbers of 
positroid subdivisions of $\mathcal{A}_{n,k,2}(Z)$
\cite[Section 11]{LPW}.  By \cref{thm:trop}, the cones of $\Trop^+ \Gr_{k+1,n}$ parameterize the 
regular positroid subdivisions of $\Delta_{k+1,n}$,
 so this leads to the idea
that positroid 
subdivisions of $\Delta_{k+1,n}$ and $\mathcal{A}_{n,k,2}(Z)$ must be
related \cite{LPW}.

\begin{example}
Continuing \cref{ex:trop1}, there are two maximal cones of 
 $\Trop^+\Gr_{2,4}$, defined by the inequalities
$P_{13}+P_{24} = P_{14}+P_{23} < P_{12}+P_{34}$ and 
 $P_{13}+P_{24} = P_{12}+P_{34} < P_{14}+P_{23}.$
These two cones give rise to the two subdivisions of 
$\Delta_{2,4}$ shown in 
\cref{fig:T1}. 
These subdivisions are both positroid tilings for  
the moment map, as in 
\cref{def:tri} (and there are no other positroid tilings).
	\begin{figure}[h]
		\centering
	\begin{tikzpicture} 
	\draw (0,1.5) circle (.5cm);
	\node[draw=none, minimum size=1cm, regular polygon, regular polygon sides=4] 
		at (0,1.5) (s) {};
	\foreach \x in {1,2,...,4}
	\fill (s.corner \x) circle[radius=1.2pt];
	 \node[shift=(s.corner 1),anchor=west] {\tiny $2$};
	 \node[shift=(s.corner 2),anchor=east] {\tiny $1$};
	 \node[shift=(s.corner 3),anchor=east] {\tiny $4$};
	 \node[shift=(s.corner 4),anchor=west] {\tiny $3$};
        \draw[black,thick](0,1.5-.1)--(s.corner 3);
        \draw[black,thick](0,1.5-.1)--(s.corner 4);
        \draw[black,thick](0,1.5+.1)--(s.corner 1);
        \draw[black,thick](0,1.5+.1)--(s.corner 2);
        \draw[black,thick](0,1.5+.1)--(0,1.5-.1);
	\filldraw[color=black,fill=black] (0,1.5+.1) circle (1pt);
	\filldraw[color=black,fill=white] (0,1.5-.1) circle (1pt);
	\filldraw[color=black,fill=white] (-.15,1.5+.2) circle (1pt);
	\filldraw[color=black,fill=white] (.15,1.5+.2) circle (1pt);
	\draw (1,1.5) node {$\leadsto$};
		\draw (-1.3,1.5) node {\tiny $(3,1,4,2)$};
		\draw (-1.3,0) node {\tiny $(2,4,1,3)$};
		\draw (0,0) circle (.5cm);
	\node[draw=none, minimum size=1cm, regular polygon, regular polygon sides=4] (s) {};
	\foreach \x in {1,2,...,4}
	\fill (s.corner \x) circle[radius=1.2pt];
	 \node[shift=(s.corner 1),anchor=west] {\tiny $2$};
	 \node[shift=(s.corner 2),anchor=east] {\tiny $1$};
	 \node[shift=(s.corner 3),anchor=east] {\tiny $4$};
	 \node[shift=(s.corner 4),anchor=west] {\tiny $3$};
        \draw[black,thick](0,-.1)--(s.corner 3);
        \draw[black,thick](0,-.1)--(s.corner 4);
        \draw[black,thick](0,.1)--(s.corner 1);
        \draw[black,thick](0,.1)--(s.corner 2);
        \draw[black,thick](0,.1)--(0,-.1);
	\filldraw[color=black,fill=white] (0,.1) circle (1pt);
	\filldraw[color=black,fill=black] (0,-.1) circle (1pt);
	\filldraw[color=black,fill=white] (-.15,-.2) circle (1pt);
	\filldraw[color=black,fill=white] (.15,-.2) circle (1pt);
	\draw (1,0) node {$\leadsto$};
\end{tikzpicture} \hspace{-.9cm}
	\begin{tikzpicture}
     \tkzDefPoint(0,0){e14}
     \tkzDefPoint(1,0){e13}
     \tkzDefPoint(.5,.4){e24}
     \tkzDefPoint(1.5,.4){e23}
     \tkzDefPoint(.7,1.4){e12}
     \tkzDefPoint(.7,-1){e34}
	 \filldraw[fill=gray!30!white](e14)--(e24)--(e23)--(e13)--(e14);
\tkzDrawPolygon(e24, e14,e13,e23,e24,e12,e14,e13,e12,e23,e34,e14,e13,e34,e23)
		\tkzDrawPolygon(e24,e14,e12,e24,e23,e12,e13,e14,e13,e23);
		\node[shift=(e24), anchor=north] {\tiny $e_{24}$};
		\node[shift=(e14), anchor= east] {\tiny $e_{14}$};
		\node[shift=(e23), anchor=west] {\tiny $e_{23}$};
		\node[shift=(e13), anchor=north east] {\tiny $e_{13}$};
		\node[shift=(e12), anchor=east] {\tiny $e_{12}$};
		\node[shift=(e34), anchor=east] {\tiny $e_{34}$};
	\filldraw[color=black,fill=black] (e24) circle (1pt);
	\filldraw[color=black,fill=black] (e14) circle (1pt);
	\filldraw[color=black,fill=black] (e23) circle (1pt);
	\filldraw[color=black,fill=black] (e13) circle (1pt);
	\filldraw[color=black,fill=black] (e12) circle (1pt);
	\filldraw[color=black,fill=black] (e34) circle (1pt);
\end{tikzpicture} \quad \quad \quad
	\begin{tikzpicture}
		\draw (-1.3,0) circle (.5cm);
	\node[draw=none, minimum size=1cm, regular polygon, regular polygon sides=4] 
		at (-1.3,0) (s) {};
	\foreach \x in {1,2,...,4}
	\fill (s.corner \x) circle[radius=1.2pt];
	 \node[shift=(s.corner 1),anchor=west] {\tiny $2$};
	 \node[shift=(s.corner 2),anchor=east] {\tiny $1$};
	 \node[shift=(s.corner 3),anchor=east] {\tiny $4$};
	 \node[shift=(s.corner 4),anchor=west] {\tiny $3$};
        \draw[black,thick](-.1-1.3,0)--(s.corner 3);
        \draw[black,thick](.1-1.3,0)--(s.corner 4);
        \draw[black,thick](.1-1.3,0)--(s.corner 1);
        \draw[black,thick](-.1-1.3,0)--(s.corner 2);
        \draw[black,thick](.1-1.3,0)--(-.1-1.3,0);
	\filldraw[color=black,fill=white] (.1-1.3,0) circle (1pt);
	\filldraw[color=black,fill=black] (-.1-1.3,0) circle (1pt);
	\filldraw[color=black,fill=white] (-.2-1.3,-.15) circle (1pt);
	\filldraw[color=black,fill=white] (-.2-1.3, .15) circle (1pt);
	\draw (-.5,0) node {$\leadsto$};
		\draw (-1.3,-.8) node {\tiny $(4,3,1,2)$};
		\draw (2.5,-.8) node {\tiny $(3,4,2,1)$};
     \tkzDefPoint(0,0){e14}
     \tkzDefPoint(1,0){e13}
     \tkzDefPoint(.5,.4){e24}
     \tkzDefPoint(1.5,.4){e23}
     \tkzDefPoint(.7,1.4){e12}
     \tkzDefPoint(.7,-1){e34}
	 \filldraw[fill=gray!30!white](e12)--(e24)--(e34)--(e13)--(e12);
\tkzDrawPolygon(e24, e14,e13,e23,e24,e12,e14,e13,e12,e23,e34,e14,e13,e34,e23)
		\tkzDrawPolygon(e24,e14,e12,e24,e23,e12,e13,e14,e13,e23);
		\node[shift=(e24), anchor=north] {\tiny $e_{24}$};
		\node[shift=(e14), anchor= north ] {\tiny $e_{14}$};
		\node[shift=(e23), anchor=south west] {\tiny $e_{23}$};
		\node[shift=(e13), anchor=north east] {\tiny $e_{13}$};
		\node[shift=(e12), anchor=east] {\tiny $e_{12}$};
		\node[shift=(e34), anchor=east] {\tiny $e_{34}$};
	\filldraw[color=black,fill=black] (e24) circle (1pt);
	\filldraw[color=black,fill=black] (e14) circle (1pt);
	\filldraw[color=black,fill=black] (e23) circle (1pt);
	\filldraw[color=black,fill=black] (e13) circle (1pt);
	\filldraw[color=black,fill=black] (e12) circle (1pt);
	\filldraw[color=black,fill=black] (e34) circle (1pt);
		\draw (2.5,0) circle (.5cm);
	\node[draw=none, minimum size=1cm, regular polygon, regular polygon sides=4] 
		at (2.5,0) (s) {};
	\foreach \x in {1,2,...,4}
	\fill (s.corner \x) circle[radius=1.2pt];
	 \node[shift=(s.corner 1),anchor=west] {\tiny $2$};
	 \node[shift=(s.corner 2),anchor=east] {\tiny $1$};
	 \node[shift=(s.corner 3),anchor=east] {\tiny $4$};
	 \node[shift=(s.corner 4),anchor=west] {\tiny $3$};
        \draw[black,thick](-.1+2.5,0)--(s.corner 3);
        \draw[black,thick](.1+2.5,0)--(s.corner 4);
        \draw[black,thick](.1+2.5,0)--(s.corner 1);
        \draw[black,thick](-.1+2.5,0)--(s.corner 2);
        \draw[black,thick](.1+2.5,0)--(-.1+2.5,0);
	\filldraw[color=black,fill=black] (.1+2.5,0) circle (1pt);
	\filldraw[color=black,fill=white] (-.1+2.5,0) circle (1pt);
	\filldraw[color=black,fill=white] (.2+2.5,-.15) circle (1pt);
	\filldraw[color=black,fill=white] (.2+2.5, .15) circle (1pt);
		\draw (2.5-.7,0) node {\rotatebox{180}{$\leadsto$}};
\end{tikzpicture}
	\caption{The two positroid tilings of 
	$\Delta_{2,4}$ for the moment map.
	The plabic graphs specify the positroid cells
	whose images are the positroid tiles (positroid polytopes).}
        \label{fig:T1}
\end{figure}
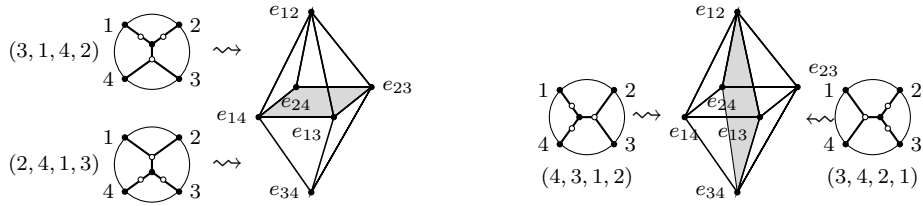

	Meanwhile, the amplituhedron $\A_{4,1,2}(Z)$ is a 
	quadrilateral, which has precisely two positroid  tilings
	for the amplituhedron map, as shown in 
	\cref{fig:T2}.
\end{example}

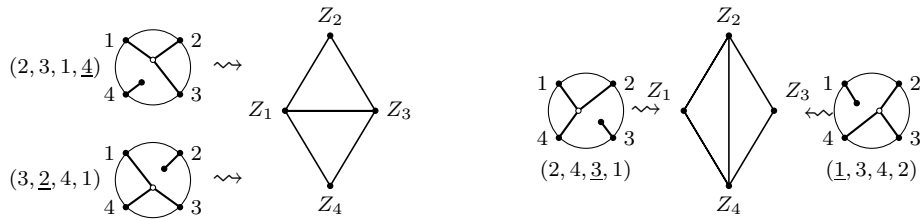
\begin{figure}[h]
	\centering
	\begin{tikzpicture} 
	\draw (0,1.5) circle (.5cm);
	\node[draw=none, minimum size=1cm, regular polygon, regular polygon sides=4] 
		at (0,1.5) (s) {};
	\foreach \x in {1,2,...,4}
	\fill (s.corner \x) circle[radius=1.2pt];
	 \node[shift=(s.corner 1),anchor=west] {\tiny $2$};
	 \node[shift=(s.corner 2),anchor=east] {\tiny $1$};
	 \node[shift=(s.corner 3),anchor=east] {\tiny $4$};
	 \node[shift=(s.corner 4),anchor=west] {\tiny $3$};
        \draw[black,thick](-.15,1.5-.2)--(s.corner 3);
        \draw[black,thick](0,1.5+.1)--(s.corner 4);
        \draw[black,thick](0,1.5+.1)--(s.corner 1);
        \draw[black,thick](0,1.5+.1)--(s.corner 2);
	\filldraw[color=black,fill=white] (0,1.5+.1) circle (1pt);
	\filldraw[color=black,fill=black] (-.15,1.5-.2) circle (1pt);
	\draw (1,1.5) node {$\leadsto$};
		\draw (-1.3,1.5) node {\tiny $(2,3,1,\underline{4})$};
		\draw (-1.3,0) node {\tiny $(3,\underline{2},4,1)$};
		\draw (0,0) circle (.5cm);
	\node[draw=none, minimum size=1cm, regular polygon, regular polygon sides=4] (s) {};
	\foreach \x in {1,2,...,4}
	\fill (s.corner \x) circle[radius=1.2pt];
	 \node[shift=(s.corner 1),anchor=west] {\tiny $2$};
	 \node[shift=(s.corner 2),anchor=east] {\tiny $1$};
	 \node[shift=(s.corner 3),anchor=east] {\tiny $4$};
	 \node[shift=(s.corner 4),anchor=west] {\tiny $3$};
        \draw[black,thick](0,-.1)--(s.corner 3);
        \draw[black,thick](0,-.1)--(s.corner 4);
        \draw[black,thick](.15,.15)--(s.corner 1);
        \draw[black,thick](0,-.1)--(s.corner 2);
	\filldraw[color=black,fill=black] (.15,.15) circle (1pt);
	\filldraw[color=black,fill=white] (0,-.1) circle (1pt);
	\draw (1,0) node {$\leadsto$};
\end{tikzpicture} \hspace{-.5cm}
	\begin{tikzpicture}
     \tkzDefPoint(0,0){Z1}
     \tkzDefPoint(1.2,0){Z3}
     \tkzDefPoint(.6,1){Z2}
     \tkzDefPoint(.6,-1){Z4}
		\tkzDrawPolygon(Z1, Z2, Z3, Z4, Z1, Z3)
		\node[shift=(Z1), anchor=east] {\tiny $Z_1$};
		\node[shift=(Z2), anchor=south] {\tiny $Z_2$};
		\node[shift=(Z3), anchor=west] {\tiny $Z_3$};
		\node[shift=(Z4), anchor=north] {\tiny $Z_4$};
	\filldraw[color=black,fill=black] (Z1) circle (1pt);
	\filldraw[color=black,fill=black] (Z2) circle (1pt);
	\filldraw[color=black,fill=black] (Z3) circle (1pt);
	\filldraw[color=black,fill=black] (Z4) circle (1pt);
\end{tikzpicture} \quad \quad \quad
	\begin{tikzpicture}
		\draw (-1.3,0) circle (.5cm);
	\node[draw=none, minimum size=1cm, regular polygon, regular polygon sides=4] 
		at (-1.3,0) (s) {};
	\foreach \x in {1,2,...,4}
	\fill (s.corner \x) circle[radius=1.2pt];
	 \node[shift=(s.corner 1),anchor=west] {\tiny $2$};
	 \node[shift=(s.corner 2),anchor=east] {\tiny $1$};
	 \node[shift=(s.corner 3),anchor=east] {\tiny $4$};
	 \node[shift=(s.corner 4),anchor=west] {\tiny $3$};
        \draw[black,thick](-.1-1.3,0)--(s.corner 3);
        \draw[black,thick](.2-1.3,-.15)--(s.corner 4);
        \draw[black,thick](-.1-1.3,0)--(s.corner 1);
        \draw[black,thick](-.1-1.3,0)--(s.corner 2);
	\filldraw[color=black,fill=black] (.2-1.3,-.15) circle (1pt);
	\filldraw[color=black,fill=white] (-.1-1.3,0) circle (1pt);
	\draw (-.5,0) node {$\leadsto$};
		\draw (-1.3,-.8) node {\tiny $(2,4,\underline{3},1)$};
		\draw (2.5,-.8) node {\tiny $(\underline{1},3,4,2)$};
     \tkzDefPoint(0,0){Z1}
     \tkzDefPoint(1.2,0){Z3}
     \tkzDefPoint(.6,1){Z2}
     \tkzDefPoint(.6,-1){Z4}
		\tkzDrawPolygon(Z1, Z2, Z3, Z4, Z1, Z2,Z4)
		\node[shift=(Z1), anchor=south east] {\tiny $Z_1$};
		\node[shift=(Z2), anchor=south] {\tiny $Z_2$};
		\node[shift=(Z3), anchor=south west] {\tiny $Z_3$};
		\node[shift=(Z4), anchor=north] {\tiny $Z_4$};
	\filldraw[color=black,fill=black] (Z1) circle (1pt);
	\filldraw[color=black,fill=black] (Z2) circle (1pt);
	\filldraw[color=black,fill=black] (Z3) circle (1pt);
	\filldraw[color=black,fill=black] (Z4) circle (1pt);
		\draw (2.5,0) circle (.5cm);
	\node[draw=none, minimum size=1cm, regular polygon, regular polygon sides=4] 
		at (2.5,0) (s) {};
	\foreach \x in {1,2,...,4}
	\fill (s.corner \x) circle[radius=1.2pt];
	 \node[shift=(s.corner 1),anchor=west] {\tiny $2$};
	 \node[shift=(s.corner 2),anchor=east] {\tiny $1$};
	 \node[shift=(s.corner 3),anchor=east] {\tiny $4$};
	 \node[shift=(s.corner 4),anchor=west] {\tiny $3$};
        \draw[black,thick](.1+2.5,0)--(s.corner 3);
        \draw[black,thick](.1+2.5,0)--(s.corner 4);
        \draw[black,thick](.1+2.5,0)--(s.corner 1);
        \draw[black,thick](-.2+2.5,.1)--(s.corner 2);
	\filldraw[color=black,fill=white] (.1+2.5,0) circle (1pt);
	\filldraw[color=black,fill=black] (-.2+2.5,.1) circle (1pt);
		\draw (2.5-.7,0) node {\rotatebox{180}{$\leadsto$}};
\end{tikzpicture}
	\caption{The two positroid tilings of 
	$\mathcal{A}_{4,1,2}(Z)$ for $\tilde{Z}$. 
	The plabic graphs specify the positroid cells
	whose images are the positroid tiles (Grasstopes).}
	\label{fig:T2}
\end{figure}

\begin{definition}\label{hatmap}
	Let  $\pi=(a_1, a_2,\dots ,a_n)$ 
	be a loopless decorated permutation (written
   in one-line notation). Its \emph{T-dual} decorated permutation is
    $\hat{\pi} : i \mapsto \pi(i-1)$, so that
	$\hat{\pi} = (a_n, a_1, a_2,\dots ,a_{n-1})$.
	Any fixed points in $\hat{\pi}$ are declared to be 
	loops.\footnote{Note that 
	our use of the `hat' notation here is 
	unrelated to the one from \cref{rem:twisted}.}
\end{definition}

For example, the four permutations $(3,1,4,2)$, $(2,4,1,3)$,
$(4,3,1,2)$, $(3,4,2,1)$ labeling the positroid tilings
of $\Delta_{2,4}$ in \cref{fig:T1}  are loopless.  Their 
T-dual images are $(2,3,1,\underline{4})$, $(3,\underline{2},4,1)$,
$(2,4,\underline{3},1)$, and $(\underline{1},3,4,2)$ -- precisely
the permutations labeling the positroid tilings of $\mathcal{A}_{4,1,2}(Z)$ in 
\cref{fig:T2}!

The T-duality map appears in 
 \cite{karp:2017ouj, LPW,
 posquotients,
 GalCritVar}, 
and is a version of an $m=4$ map from \cite{abcgpt}.

\begin{prop}
	[{\cite[Lemma 5.2]{LPW} and 
	\cite[Proposition 8.1]{PSBW}}]
	\label{prop:posetiso}
      T-duality is a bijection 
	between
       loopless cells of $Gr^{\geq 0}_{k+1,n}$
       and coloopless cells of $\Grk$.  
	Moreover it is a poset isomorphism: we have 
	$S_\mu \subset \overline{S_\pi}$ if and only if $S_{\td{\mu}} \subset \overline{S_{\td{\pi}}}$.
\end{prop}

One can also describe T-duality as a map on reduced 
plabic graphs $G$; we say 
$G$ is \emph{black-trivalent} (\emph{white-trivalent}) if all of its interior black (white) vertices are trivalent.
The following construction 
streamlines the bijection of \cite[Proposition 7.15]{GalashinPostWilliams} and
\cite[Proposition 8.3]{GalCritVar} (by avoiding an intermediate step
involving plabic and zonotopal tilings).

\begin{definition}[{\cite[Definition 8.6]{PSBW}}] 
	\label{defn:plabicTDual} Let $G$ be a reduced black-trivalent plabic graph. The \emph{T-dual} of $G$, denoted $\hat{G}$, is the graph obtained as follows
	(see \cref{fig:plabicduality}).  
\begin{itemize}
\item In each face $f$ of $G$, place a black vertex $\hat{b}(f)$.
\item``On top of" each black vertex $b$ of $G$, place a white vertex $\hat{w}(b)$;
\item For each black vertex $b$ of $G$ incident to face $f$, add edge 
	$(\hat{w}(b), \hat{b}(f))$;
\item Put $\hat{i}$ on the boundary of $G$ between vertices $i-1$ and $i$ and draw an edge from $\hat{i}$ to $\hat{b}(f)$, where $f$ is the adjacent boundary face.
\end{itemize}
\end{definition}
 Note that
 the plabic graphs in \cref{fig:T2} are obtained
 from those 
 in \cref{fig:T1}   by T-duality.


\begin{figure}[h]\centering
	\resizebox{1.6in}{!}{
	\begin{tikzpicture}
		\draw (0,0) circle (2.15cm);
	 \draw[color=black, thin] (-1.4,.3)--(-.8,.8)--(0,.8)--(0.5,.3)--(0.8,-.4)--(0.3,-1.1)--(-.5,-1.35);
	 \draw[color=black, thin] (-2.13,-.3)--(-1.4,.3)--(-1.9,1);
	\filldraw[color=black,fill=black] (-1.4,.3) circle (2pt);
	\draw[color=black,thin] (-.8,.8)--(-.5,2.1);
	\filldraw[color=black,fill=black] (-.8,.8) circle (2pt);
	\draw[color=black,thin] (0,.8)--(.73,2.03);
	\filldraw[color=black,fill=white] (0,.8) circle (2pt);
	\draw[color=black,thin] (.5,.3)--(1.8,1.2);
	\filldraw[color=black,fill=black] (0.5,.3) circle (2pt);
	\draw[color=black,thin] (.8,-.4)--(2.15,-.1);
	\filldraw[color=black,fill=black] (0.8,-.4) circle (2pt);
	\draw[color=black,thin] (.3,-1.1)--(1.5,-1.57);
	\filldraw[color=black,fill=black] (0.3,-1.1) circle (2pt);
	\draw[color=black,thin] (0.2,-2.15)--(-.5,-1.35)--(-1.3,-1.7);
	\filldraw[color=black,fill=white] (-.5,-1.35) circle (2pt);
	\draw[color=black] (1.9, 1.3) node {{ $2$}};
	\draw[color=black] (2.3, -.1) node {{ $3$}};
	\draw[color=black] (1.6, -1.6) node {{ $4$}};
	\draw[color=black] (0.1, -2.4) node {{ $5$}};
	\draw[color=black] (-1.6, -1.8) node {{ $6$}};
	\draw[color=black] (-2.3, -.4) node {{ $7$}};
	\draw[color=black] (-2.1, 1.2) node {{ $8$}};
	\draw[color=black] (-.6, 2.3) node {{ $9$}};
	\draw[color=black] (.7, 2.2) node {{ $1$}};
	\end{tikzpicture} \quad \quad
}
\resizebox{1.4in}{!}{
	\begin{tikzpicture}
		\draw (0,0) circle (2.15cm);
	\node[draw=none, minimum size=4.3cm, regular polygon, regular polygon sides=9] (s) {};
	\node[draw=none, minimum size=3.6cm, regular polygon, regular polygon sides=9] (a) {};
	\foreach \x in {1,2,...,9}
		\fill[color=blue] (s.corner \x) circle[radius=1.2pt];
	\foreach \x in {1,2,...,9}
		\fill[color=blue] (a.corner \x) circle[radius=3pt];
	\foreach \x in {1,2,...,9}
        \draw[blue,thick](a.corner \x)--(s.corner \x);
		\node[color=blue, shift=(s.corner 1), anchor=south] {$\hat{1}$};
		\node[color=blue, shift=(s.corner 2),anchor=south] {$\hat{9}$};
		\node[color=blue,shift=(s.corner 3),anchor=east] {$\hat{8}$};
		\node[color=blue, shift=(s.corner 4),anchor=east] {$\hat{7}$};
		\node[color=blue, shift=(s.corner 5),anchor=north] {$\hat{6}$};
		\node[color=blue, shift=(s.corner 6),anchor=north] {$\hat{5}$};
		\node[color=blue, shift=(s.corner 7),anchor=west] {$\hat{4}$};
		\node[color=blue, shift=(s.corner 8),anchor=west] {$\hat{3}$};
		\node[color=blue, shift=(s.corner 9),anchor=south west] {$\hat{2}$};
        \draw[blue,thick](-.8,.8)--(a.corner 1);
        \draw[blue,thick](-.8,.8)--(a.corner 2);
        \draw[blue,thick](-.8,.8)--(a.corner 4);
        \draw[blue,thick](-1.4,.3)--(a.corner 2);
        \draw[blue,thick](-1.4,.3)--(a.corner 3);
        \draw[blue,thick](-1.4,.3)--(a.corner 4);
        \draw[blue,thick](0.5,.3)--(a.corner 4);
        \draw[blue,thick](0.5,.3)--(a.corner 9);
        \draw[blue,thick](0.5,.3)--(a.corner 8);
        \draw[blue,thick](0.3,-1.1)--(a.corner 4);
        \draw[blue,thick](0.3,-1.1)--(a.corner 6);
        \draw[blue,thick](0.3,-1.1)--(a.corner 7);
        \draw[blue,thick](0.8,-.4)--(a.corner 4);
        \draw[blue,thick](0.8,-.4)--(a.corner 7);
        \draw[blue,thick](0.8,-.4)--(a.corner 8);
	\filldraw[color=blue,fill=white] (-.8,.8) circle (4.5pt);
	\filldraw[color=blue,fill=white] (-1.4,.3) circle (4.5pt);
	\filldraw[color=blue,fill=white] (0.5,.3) circle (4.5pt);
	\filldraw[color=blue,fill=white] (0.3,-1.1) circle (4.5pt);
	\filldraw[color=blue,fill=white] (0.8,-.4) circle (4.5pt);
	\filldraw[color=blue,fill=white] (-.8,.8) circle (4pt);
	\filldraw[color=blue,fill=white] (-1.4,.3) circle (4pt);
	\filldraw[color=blue,fill=white] (0.5,.3) circle (4pt);
	\filldraw[color=blue,fill=white] (0.3,-1.1) circle (4pt);
	\filldraw[color=blue,fill=white] (0.8,-.4) circle (4pt);
	 \draw[color=black, thin] (-1.4,.3)--(-.8,.8)--(0,.8)--(0.5,.3)--(0.8,-.4)--(0.3,-1.1)--(-.5,-1.35);
	 \draw[color=black, thin] (-2.13,-.3)--(-1.4,.3)--(-1.9,1);
	\filldraw[color=black,fill=black] (-1.4,.3) circle (1.5pt);
	\draw[color=black,thin] (-.8,.8)--(-.5,2.1);
	\filldraw[color=black,fill=black] (-.8,.8) circle (1.5pt);
	\draw[color=black,thin] (0,.8)--(.73,2.03);
	\filldraw[color=black,fill=white] (0,.8) circle (1.5pt);
	\draw[color=black,thin] (.5,.3)--(1.8,1.2);
	\filldraw[color=black,fill=black] (0.5,.3) circle (1.5pt);
	\draw[color=black,thin] (.8,-.4)--(2.15,-.1);
	\filldraw[color=black,fill=black] (0.8,-.4) circle (1.5pt);
	\draw[color=black,thin] (.3,-1.1)--(1.5,-1.57);
	\filldraw[color=black,fill=black] (0.3,-1.1) circle (1.5pt);
	\draw[color=black,thin] (0.2,-2.15)--(-.5,-1.35)--(-1.3,-1.7);
	\filldraw[color=black,fill=white] (-.5,-1.35) circle (1.5pt);
	\draw[color=black] (1.9, 1.3) node {{ $2$}};
	\draw[color=black] (2.3, -.1) node {{ $3$}};
	\draw[color=black] (1.6, -1.6) node {{ $4$}};
	\draw[color=black] (0.1, -2.4) node {{ $5$}};
	\draw[color=black] (-1.6, -1.8) node {{ $6$}};
	\draw[color=black] (-2.3, -.4) node {{ $7$}};
	\draw[color=black] (-2.1, 1.2) node {{ $8$}};
	\draw[color=black] (-.6, 2.3) node {{ $9$}};
	\draw[color=black] (.7, 2.2) node {{ $1$}};
\end{tikzpicture}
}
	\caption{ At left: a plabic graph $G$ with trip permutation
	$\pi_G = (5,9,2,3,6,4,1,7,8)$.  At right: $G$ with the T-dual graph
	$\hatG$ superimposed.  We have $\pi_{\hatG} = 
	 (8,5,9,2,3,\underline{6},4,1,7)$. Note that $\hatG$ is the 
	 graph from 
	 \cref{fig:plabic}.}
        \label{fig:plabicduality}
\end{figure}
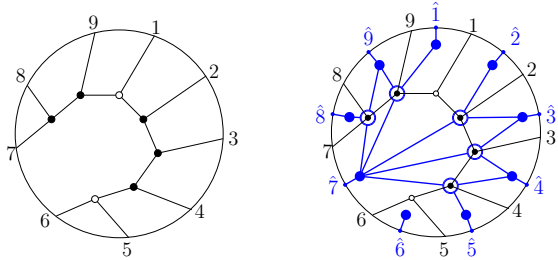

\begin{prop} 
	[{\cite[Proposition 8.8]{PSBW}}]
	\label{prop:TdualityPlabicGraphs} 
	Let $G$ be a reduced black-trivalent plabic graph with trip permutation $\pi_G = \pi$. Then $\hat{G}$ is a reduced white-trivalent plabic graph with  $\pi_{\hat{G}} = \hat{\pi}$.
\end{prop}

T-duality provides a link between 
positroid tilings of $\A_{n,k,2}(Z)$ and $\Delta_{k+1,n}$.
The first result is that T-duality gives a bijection
between positroid tiles of $\Delta_{k+1,n}$ (see
\cref{prop:tree}) and 
positroid tiles of 
$\mathcal{A}_{n,k,2}$ 
	(see \cref{thm:allGTs}). 

Given a bicolored triangulation $\T$,
we define  $\area(a\to b)$ to be the number of black triangles to the 
left of $a\to b$ in any triangulation of $\overline{\T}$ compatible
with $a\to b$.

\begin{theorem}[{\cite[Section 8]{PSBW}}]
	\label{cor:GTsInBijection} 
	Given a bicolored triangulation $\T$ of type $(k,n)$, we can read off
	T-dual graphs $G$ and $\hat{G}$ giving positroid tiles 
$\Gamma_G$ and $Z_{\hat{G}}$ as follows:
\begin{itemize}
	\item  $G: = G(\T)$ is the dual graph 
		of $\T$, as 
		shown at the left of \cref{fig:unpuncTduality}.
	\item  $\hat{G}:=\hat{G}(\T)$ is the graph from
		\cref{def:param}, as shown at the right of \cref{fig:unpuncTduality}.
\end{itemize}
This correspondence gives a bijection between 
positroid tiles of $\Delta_{k+1,n}$ and $\A_{n,k,2}$, both of
	which depend only on $\overline{\T}$.

Moreover, if we let $h\to j$ (with $h<j$) 
	range over arcs 
of $\T$, the inequalities 
\begin{align*}
        &\mbox{(1)} 
	& \area(h \to j)+1 \geq  x_h+x_{h+1} + \dots + x_{j-1}  \geq \area(h \to j) \quad \text{for }&x \in \Gamma_{G(\T)}&\\
        &\mbox{(2)}
	&  
	(-1)^{\area(h \to j)} 
	\langle Y Z_h Z_j\rangle \geq 0
	\quad \text{for }&Y \in \gt{\hatG(\T)}&
\end{align*}
  cut out the positroid tiles $\Gamma_{G(\T)}$ and $\gt{\hatG(\T)}$.
\end{theorem}
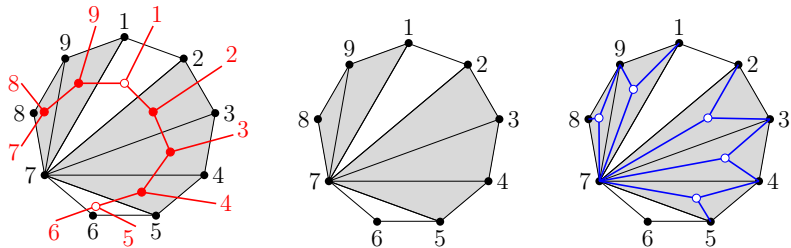
\begin{figure}[h]
\centering
\resizebox{4.3in}{!}{
\begin{tikzpicture}
	\node[draw, minimum size=3.2cm, regular polygon, regular polygon sides=9] (a) {};
	 \filldraw[fill=gray!30!white](a.corner 1)--(a.corner 2)--(a.corner 3)--(a.corner 4)--(a.corner 1);
	 \filldraw[fill=gray!30!white](a.corner 4)--(a.corner 6)--(a.corner 7)--(a.corner 8)--(a.corner 9)--(a.corner 4);
	\foreach \x in {1,2,...,9}
	\fill (a.corner \x) circle[radius=2pt];
	\draw[black,thin](a.corner 1)--(a.corner 4);
	\draw[black,thin](a.corner 9)--(a.corner 4);
	\draw[black,thin](a.corner 6)--(a.corner 4);
	\draw[black,thin](a.corner 2)--(a.corner 4);
	\draw[black,thin](a.corner 7)--(a.corner 4);
	\draw[black,thin](a.corner 8)--(a.corner 4);
	 \node[shift=(a.corner 1), anchor=south] {$1$};
	 \node[shift=(a.corner 2),anchor=south] {$9$};
	 \node[shift=(a.corner 3),anchor=east] {$8$};
	 \node[shift=(a.corner 4),anchor=east] {$7$};
	 \node[shift=(a.corner 5),anchor=north] {$6$};
	 \node[shift=(a.corner 6),anchor=north] {$5$};
	 \node[shift=(a.corner 7),anchor=west] {$4$};
	 \node[shift=(a.corner 8),anchor=west] {$3$};
	 \node[shift=(a.corner 9),anchor=west] {$2$};
	 \draw[color=red, thick] (-1.4,.3)--(-.8,.8)--(0,.8)--(0.5,.3)--(0.8,-.4)--(0.3,-1.1)--(-.5,-1.35);
	 \draw[color=red, thick] (-1.8,-.3)--(-1.4,.3)--(-1.8,.8);
	\filldraw[color=red,fill=red] (-1.4,.3) circle (2pt);
	\draw[color=red,thick] (-.8,.8)--(-.5,1.8);
	\filldraw[color=red,fill=red] (-.8,.8) circle (2pt);
	\draw[color=red,thick] (0,.8)--(.5,1.8);
	\filldraw[color=red,fill=white] (0,.8) circle (2pt);
	\draw[color=red,thick] (.5,.3)--(1.7,1.1);
	\filldraw[color=red,fill=red] (0.5,.3) circle (2pt);
	\draw[color=red,thick] (.8,-.4)--(1.9,-.1);
	\filldraw[color=red,fill=red] (0.8,-.4) circle (2pt);
	\draw[color=red,thick] (.3,-1.1)--(1.6,-1.3);
	\filldraw[color=red,fill=red] (0.3,-1.1) circle (2pt);
	\draw[color=red,thick] (0.2,-1.7)--(-.5,-1.35)--(-1.2,-1.6);
	\filldraw[color=red,fill=white] (-.5,-1.35) circle (2pt);
	\draw[color=red] (1.8, 1.3) node {{ $2$}};
	\draw[color=red] (2, 0) node {{ $3$}};
	\draw[color=red] (1.7, -1.3) node {{ $4$}};
	\draw[color=red] (0, -1.9) node {{ $5$}};
	\draw[color=red] (-1.3, -1.8) node {{ $6$}};
	\draw[color=red] (-2, -.4) node {{ $7$}};
	\draw[color=red] (-2, .8) node {{ $8$}};
	\draw[color=red] (-.6, 2) node {{ $9$}};
	\draw[color=red] (.5, 2) node {{ $1$}};
\end{tikzpicture}
	\hspace{.3cm}
\begin{tikzpicture}
	\node[draw, minimum size=3.2cm, regular polygon, regular polygon sides=9] (a) {};
	 \filldraw[fill=gray!30!white](a.corner 1)--(a.corner 2)--(a.corner 3)--(a.corner 4)--(a.corner 1);
	 \filldraw[fill=gray!30!white](a.corner 4)--(a.corner 6)--(a.corner 7)--(a.corner 8)--(a.corner 9)--(a.corner 4);
	\foreach \x in {1,2,...,9}
	\fill (a.corner \x) circle[radius=2pt];
	\draw[black,thin](a.corner 1)--(a.corner 4);
	\draw[black,thin](a.corner 9)--(a.corner 4);
	\draw[black,thin](a.corner 6)--(a.corner 4);
	\draw[black,thin](a.corner 2)--(a.corner 4);
	\draw[black,thin](a.corner 7)--(a.corner 4);
	\draw[black,thin](a.corner 8)--(a.corner 4);
	 \node[shift=(a.corner 1), anchor=south] {$1$};
	 \node[shift=(a.corner 2),anchor=south] {$9$};
	 \node[shift=(a.corner 3),anchor=east] {$8$};
	 \node[shift=(a.corner 4),anchor=east] {$7$};
	 \node[shift=(a.corner 5),anchor=north] {$6$};
	 \node[shift=(a.corner 6),anchor=north] {$5$};
	 \node[shift=(a.corner 7),anchor=west] {$4$};
	 \node[shift=(a.corner 8),anchor=west] {$3$};
	 \node[shift=(a.corner 9),anchor=west] {$2$};
\end{tikzpicture}
	\hspace{.3cm}
\begin{tikzpicture}
	\node[draw, minimum size=3.2cm, regular polygon, regular polygon sides=9] (a) {};
	 \filldraw[fill=gray!30!white](a.corner 1)--(a.corner 2)--(a.corner 3)--(a.corner 4)--(a.corner 1);
	 \filldraw[fill=gray!30!white](a.corner 4)--(a.corner 6)--(a.corner 7)--(a.corner 8)--(a.corner 9)--(a.corner 4);
	\foreach \x in {1,2,...,9}
	\fill (a.corner \x) circle[radius=2pt];
	\draw[black,thin](a.corner 1)--(a.corner 4);
	\draw[black,thin](a.corner 9)--(a.corner 4);
	\draw[black,thin](a.corner 6)--(a.corner 4);
	\draw[black,thin](a.corner 2)--(a.corner 4);
	\draw[black,thin](a.corner 7)--(a.corner 4);
	\draw[black,thin](a.corner 8)--(a.corner 4);
	 \node[shift=(a.corner 1), anchor=south] {$1$};
	 \node[shift=(a.corner 2),anchor=south] {$9$};
	 \node[shift=(a.corner 3),anchor=east] {$8$};
	 \node[shift=(a.corner 4),anchor=east] {$7$};
	 \node[shift=(a.corner 5),anchor=north] {$6$};
	 \node[shift=(a.corner 6),anchor=north] {$5$};
	 \node[shift=(a.corner 7),anchor=west] {$4$};
	 \node[shift=(a.corner 8),anchor=west] {$3$};
	 \node[shift=(a.corner 9),anchor=west] {$2$};
        \draw[blue,thick](-.8,.8)--(a.corner 1);
        \draw[blue,thick](-.8,.8)--(a.corner 2);
        \draw[blue,thick](-.8,.8)--(a.corner 4);
        \draw[blue,thick](-1.4,.3)--(a.corner 2);
        \draw[blue,thick](-1.4,.3)--(a.corner 3);
        \draw[blue,thick](-1.4,.3)--(a.corner 4);
        \draw[blue,thick](0.5,.3)--(a.corner 4);
        \draw[blue,thick](0.5,.3)--(a.corner 9);
        \draw[blue,thick](0.5,.3)--(a.corner 8);
        \draw[blue,thick](0.3,-1.1)--(a.corner 4);
        \draw[blue,thick](0.3,-1.1)--(a.corner 6);
        \draw[blue,thick](0.3,-1.1)--(a.corner 7);
        \draw[blue,thick](0.8,-.4)--(a.corner 4);
        \draw[blue,thick](0.8,-.4)--(a.corner 7);
        \draw[blue,thick](0.8,-.4)--(a.corner 8);
	\filldraw[color=blue,fill=white] (-.8,.8) circle (2pt);
	\filldraw[color=blue,fill=white] (-1.4,.3) circle (2pt);
	\filldraw[color=blue,fill=white] (0.5,.3) circle (2pt);
	\filldraw[color=blue,fill=white] (0.3,-1.1) circle (2pt);
	\filldraw[color=blue,fill=white] (0.8,-.4) circle (2pt);
\end{tikzpicture}
}
        \caption{In the middle: a bicolored triangulation  $\T$, 
	with the dual graph $G(\T)$ to its left, and the T-dual
	graph $\hatG(\T)$ to its right.}
        \label{fig:unpuncTduality}
\end{figure}


We now explain how  Eulerian numbers enter the story.

\begin{defn}
Let $w \in S_n$. We call a letter $i\geq 2$ in $w$ a \emph{left descent}
	if  
 $w^{-1}(i) < w^{-1}(i-1)$.
 And we say that $i\in [n]$ in $w$ is a \emph{cyclic left descent} if either
 $i\geq 2$ is a left descent of $w$ or if 
	$i=1$ and 
	$w^{-1}(1) < w^{-1}(n)$.
Let $\cdes(w)$ denote the set of cyclic left descents of $w$.
\end{defn}

Let $D_{k+1, n}$ be the set of permutations $w \in S_n$ with $k+1$ cyclic descents and $w_n=n$. Note that $|D_{k+1, n}|$ equals
the \emph{Eulerian number} 
$E_{k,n-1}:= \sum_{\ell = 0}^{k+1} (-1)^{\ell} {n \choose \ell} (k+1-\ell)^{n-1}.$

\begin{defn} \label{defn:wsimplexHSimplex}
 For $w \in D_{k+1, n}$, let $w^{(a)}$ denote the cyclic rotation of $w$ ending at $a$.
Let $I_1=I_1(w):=\cdes(w)$ and 
	for $2 \leq r \leq n$, 
let $I_r=I_r(w):=\cdes(w^{(r-1)})$.
	We then define the \emph{w-simplex} $\simp{w}$ to be the convex hull
	$\simp{w}:=\convex(e_{I_1}, \dots, e_{I_n}) \subseteq \Delta_{k+1,n}$.
\end{defn}
\begin{example}
	If $w=(1,3,2,4)$, then we have 
	$I_1 = \{1,3\}$, $I_2 = \{2,3\}$, $I_3 = \{3,4\}$, and $I_4 = \{2,4\}$,
	so $\Delta_w = \convex(e_{13}, e_{23}, e_{34}, e_{24})$.
        See \cref{fig:T1}.
\end{example}

Stanley gave the first triangulation of the hypersimplex \cite{StanleyTriangulation}, see also 
 \cite{SturmfelsGrobner} and \cite{LamPost}.
\begin{proposition}
	[{\cite{StanleyTriangulation}}]
	\text{ We have }
	$$\Delta_{k+1,n} = \bigcup_{w\in D_{k+1,n}} \Delta_w.$$
\end{proposition}

\begin{example}
	For example, $\Delta_{24} = 
	\Delta_{1324} \cup \Delta_{2134} \cup \Delta_{2314} 
	\cup \Delta_{3124}$.  This decomposition refines 
	the two positroid tilings shown in 
        \cref{fig:T1}
	(and this holds in general \cite{LamPost}).
\end{example}

\begin{defn} 
	\label{def:ampchamber}
	Let $w \in D_{k+1,n}$ and let $I_1,\dots,I_n$ be as in 
\cref{defn:wsimplexHSimplex}.
We define 
	$\asimpo{w}$ to be the open amplituhedron
	chamber 
	consisting of $Y \in \A_{n, k, 2}(Z)$ 
	such that for $a=1, \dots, n$, the 
sign flips of the sequence
\[
	(  \langle Y Z_a \hat{Z}_{1} \rangle , \langle Y Z_a \hat{Z}_{2} \rangle, \dots, \langle Y Z_a \hat{Z}_{a-1} \rangle, \langle Y Z_a Z_a\rangle, \langle Y Z_a Z_{a+1} \rangle, \dots, \langle 
	Y Z_a Z_n \rangle)
\]
occur precisely in positions
        $I_a \setminus \{a\}$,
		where we say a sign flip occurs in position $j$ if $\langle Y Z_a Z_j \rangle$ and $\langle Y Z_a Z_{j+1} \rangle$ are nonzero and have different signs (if $j=n$ we consider $j+1=1$).

	We sometimes refer to $\asimpo{w}$ and 
	$\asimp{w}:= \overline{\asimpo{w}}$ as 
	open and closed \emph{$w$-chambers}.
\end{defn}

\begin{example}
	If $w=(1,3,2,4)$, then 
	$I_1 = \{1,3\}$, $I_2 = \{2,3\}$, $I_3 = \{3,4\}$, and $I_4 = \{2,4\}$,
	so $\asimpo{w}$ consists of $Y \in \A_{n,1,2}(Z)$
	such that  
	$\lrangle{YZ_1 Z_4}<0$,
	$\lrangle{YZ_2 Z_4}<0$, 
	and the other four
	$\lrangle{YZ_i Z_j}$ with $i<j$ are positive.
        In \cref{fig:T2}, this corresponds to the triangle
	with vertices $Z_3$, $Z_4$, and the point where
	the two diagonals of the quadrilateral intersect.
\end{example}


For some choices of $Z$, the $w$-chamber
$\asimpo{w}$ can be empty.
However 
for each $w\in D_{k+1,n}$ one can find explicit matrices 
$Z\in \Mat_{n,k+2}^{>0}$
such that $\asimpo{w}$ is nonempty \cite[Section 11]{PSBW}.
Moreover
the $w$-chambers are
the only  amplituhedron
chambers which are realizable.

\begin{thm} 
	[{\cite[Theorem 10.10]{PSBW}}]
	\label{thm:wSimpCover}
	For any $Z \in \Mat^{>0}_{n, k+2}$,
		we have 
        $$\A_{n, k, 2}(Z)= \bigcup_{w \in D_{k+1,n}} \asimp{w}.
	$$
 \end{thm}
	By \cref{cor:GTsInBijection},
	each  positroid tile
 $Z_{\hat{G}}$ is a union of closed $w$-chambers.
\begin{example}
	For example, $\A_{4,1,2}(Z) = 
	\asimp{1324} \cup \asimp{2134} \cup \asimp{2314} 
	\cup \asimp{3124}$.  This decomposition refines 
	the two positroid tilings shown in 
	\cref{fig:T2}.
\end{example}

Given that positroid tiles
$\Gamma_G \subset \Delta_{k+1,n}$ are unions of $w$-simplices,
and positroid tiles $Z_{\hat{G}} \subset \A_{n,k,2}(Z)$ are unions of $w$-chambers, the following is the key to proving that positroid
tilings of $\Delta_{k+1,n}$ and $\A_{n,k,2}(Z)$ are in bijection.

 \begin{prop}
	 [{\cite[Proposition 11.1]{PSBW}}]
	 \label{prop:simplexContainment} 
	 Let $Z \in \Mat^{>0}_{n,k+2}$. Suppose $w \in D_{k+1,n}$ and that $\asimp{w} \neq \emptyset$. For any positroid tile $\Gamma_{\pi}$, $\simp{w} \subset \Gamma_{\pi}$ if and only if $\asimp{w} \subset \gt{\td{\pi}}$.
\end{prop}

Figures \ref{fig:T1} and \ref{fig:T2} 
illustrate the fact that the two positroid 
tilings of $\mathcal{A}_{4,1,2}(Z)$ are related to the 
two positroid tilings of $\Delta_{2,4}$ by T-duality.
The following result, first conjectured in \cite[Conjecture 6.9]{LPW}, 
generalizes this example to arbitrary $k$ and $n$.

\begin{theorem} 
[{\cite{PSBW}}]
\label{conj:triangCorresp} 
  A collection $\{\Gamma_\pi\}$ of positroid polytopes in $\Delta_{k+1, n}$ gives a positroid tiling of $\Delta_{k+1, n}$ if and only if for all $Z \in \Mat_{n,k+2}^{>0}$, the collection $\{\gt{\td{\pi}}\}$ of Grasstopes gives a positroid tiling of $\A_{n, k, 2}(Z)$.
\end{theorem}

In light of the fact that $\Delta_{k+1,n}$ is an $(n-1)$-dimensional polytope,
and $\mathcal{A}_{n,k,2}(Z)$ is a $2k$-dimensional non-polytopal 
subset of $\Gr_{k,k+2}$, we find \cref{conj:triangCorresp} very surprising!

We believe that more generally, \cref{conj:triangCorresp}
extends to give a bijection between positroid dissections (respectively, subdivisions)
of $\Delta_{k+1,n}$ and positroid dissections (respectively, subdivisions) of $\A_{n,k,2}(Z)$,
see \cite[Conjectures 6.9 and 8.8]{LPW}.

Given that $\Trop^+ \Gr_{k+1,n}$ controls the regular positroid subdivisions of $\Delta_{k+1,n}$,
which are (conjecturally) in bijection with positroid subdivisions of $\A_{n,k,2}(Z)$, it is natural
to ask:  can we make a direct connection between $\Trop^+ \Gr_{k+1,n}$ and $\A_{n,k,2}(Z)$?
Is there a way to think of points of $\Trop^+\Gr_{k+1,n}$ as giving ``height functions''
for $\A_{n,k,2}(Z)$?



\section{The amplituhedron and cluster algebras}\label{sec:cluster}

\emph{Cluster algebras} are a remarkable class of commutative
rings introduced by Fomin and Zelevinsky \cite{ca1, FominICM}, see also
\cite{FWZ}.
Many homogeneous coordinate rings of ``nice'' algebraic varieties have a cluster algebra
structure, including the Grassmannian 
 \cite{Scott}.
Starting in 2013, various authors connected scattering amplitudes
of planar $\mathcal{N}=4$ 
 super Yang-Mills theory
to cluster algebras
\cite{Golden:2013xva, Drummond:2017ssj, Lukowski:2019sxw, PSBW}.
In this section we explain several connections between 
the amplituhedron $\A_{n,k,m}(Z) \subset \Gr_{k, k+m}$ and the cluster algebra structure
on $\Gr_{m,n}$.  

\subsection{Cluster adjacency for facets of positroid tiles}

{Facets} of positroid tiles
in $\mathcal{A}_{n,k,m}(Z)$ are related
to the cluster  structure on $\Gr_{m,n}$.

\begin{defn}\label{def:facetG}
 Let $\gt{\pi}$ be a Grasstope of
 $\mathcal{A}_{n,k,m}(Z)$. We say that $\gt{\pi'}$ is a 
	\emph{facet} of $\gt{\pi}$ if it is maximal by inclusion among the Grasstopes satisfying the following three properties:
$\gt{\pi'} \subset \partial \gt{\pi}$; 
the cell $S_{\pi'}$ is contained in $\overline{S_\pi}$; 
and 
$\gt{\pi'}$ has codimension 1 in $\gt{\pi}$.
\end{defn}

Recall from \cite{ca1, FWZ}
that 
the cluster variables for 
$\Gr_{2,n}$ 
are the Pl\"ucker coordinates $p_{ij}$, and a collection of 
 Pl\"ucker coordinates is \emph{compatible} if the 
corresponding diagonals in an $n$-gon are noncrossing.  When $m=2$,
we have the following theorem (whose first paragraph 
is the \emph{cluster adjacency conjecture}
from \cite{Lukowski:2019sxw}).

\begin{theorem} [{\cite[Theorem 9.12]{PSBW}}] \label{thm:adj}
  Let $Z_{\hat{G}(\T)}$ be a positroid tile of $\mathcal{A}_{n,k,2}(Z)$.
Each facet lies on a hypersurface
$\langle Y Z_ i Z_j\rangle=0$,
and the collection of Pl\"ucker coordinates
       $\{p_{ij}\}_{\hat{G}(\T)}$
       corresponding to facets is a collection of compatible cluster variables
for $\Gr_{2,n}$.

       If $p_{hl}$ is
        compatible with
       $\{p_{ij}\}_{\hat{G}(\T)}$,  then
        $\lrangle{Y Z_h Z_l}$  has a fixed sign on $Z^{\circ}_{\hat{G}(\T)}$.
\end{theorem}

For $m>2$ 
the Grassmannian $\Gr_{m,n}$ has infinitely many cluster variables,
each of which 
can be written as a polynomial
 $Q(p_I)$ in the ${n \choose m}$ Pl\"ucker coordinates.
Meanwhile, each facet of a positroid tile  of the amplituhedon $\A_{n,k,m}(Z)$
lies on a hypersurface defined by the vanishing of some 
polynomial $Q(\langle Y Z_I \rangle)$ in the ${n \choose m}$ 
twistor coordinates 
$\lrangle{Y Z_I}_{I \in {[n] \choose m}}$.  

\begin{conj}[
	{\cite[Conjecture 6.2]{PSBW}}]
	\label{conj:cluster} 
Let $Z_{\pi}$ be a positroid tile of 
	$\A_{n,k,m}(Z)$ and let
\begin{equation*}
\Facet(Z_{\pi}):=\{Q(p_I) \ \vert \ 
               \text{a facet of $Z_{\pi}$ lies on
               the hypersurface $Q(\lrangle{Y Z_I})=0$}\},
\end{equation*}
       where $Q$ is a polynomial in the ${n \choose m}$ Pl\"ucker coordinates.
 Then
		 each $Q\in \Facet(Z_{\pi})$ is a cluster variable for $\Gr_{m,n}$, and 
                 $\Facet(Z_{\pi})$ consists of compatible cluster variables.
        Moreover if $\tilde Q$ is a cluster variable compatible with
        $\Facet(Z_{\pi})$,
        the polynomial $\tilde Q(\lrangle{YZ_I})$ in twistor coordinates has a fixed sign
        on $Z^{\circ}_{\pi}$.
\end{conj}

\subsection{Positroid tiles as totally positive parts of cluster varieties}

 Following \cite[Section 6.2]{PSBW}, 
we now build a cluster variety $\mcv_{\overline{\T}}$ in $\Gr_{k, k+2}(\CC)$ for each positroid tile $\gt{\hatG(\overline{\T})}$ of $\A_{n, k, 2}(Z)$. 
The positroid tile
$\gto{\hatG(\overline{\T})}$
is exactly the \emph{totally positive part} of $\mcv_{\overline{\T}}$
(in the sense of \cite{FominICM}).

Fix a bicolored subdivision $\overline{\T}$ of type $(k,n)$, with black polygons $P_1, \dots, P_r$. 
Let $\mathcal{S}(\overline{\mathcal{T}})$ 
denote the set of all bicolored triangulations
 represented by $\overline{\mathcal{T}}$.
For each black polygon $P_i$, fix a
\emph{distinguished boundary arc}  $h_i \to j_i$ with $h_i < j_i$ in the boundary of $P_i$. 
   We will build $\mcv_{\overline{\T}}$ by defining seeds in the field of rational functions on $\Gr_{k, k+2}(\CC)$.

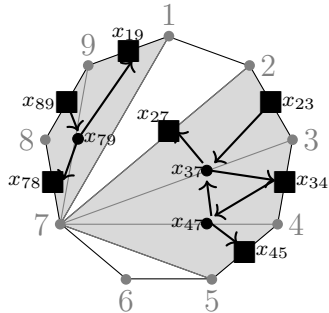
\begin{figure}[h]
	\centering
	\resizebox{1.9in}{!}{
\begin{tikzpicture}
	\node[draw, minimum size=3.3cm, regular polygon, regular polygon sides=9] (a) {};
	\node[draw=none, minimum size=3cm, regular polygon, regular polygon sides=9] (t) {};
	 \filldraw[fill=gray!30!white](a.corner 1)--(a.corner 2)--(a.corner 3)--(a.corner 4)--(a.corner 1);
	 \filldraw[fill=gray!30!white](a.corner 4)--(a.corner 6)--(a.corner 7)--(a.corner 8)--(a.corner 9)--(a.corner 4);
	\foreach \x in {1,2,...,9}
	\fill[fill=gray] (a.corner \x)  circle[radius=2pt];
	\draw[gray,thin](a.corner 1)--(a.corner 4);
	\draw[gray,thin](a.corner 9)--(a.corner 4);
	\draw[gray,thin](a.corner 6)--(a.corner 4);
	\draw[gray,thin](a.corner 2)--(a.corner 4);
	\draw[gray,thin](a.corner 7)--(a.corner 4);
	\draw[gray,thin](a.corner 8)--(a.corner 4);
	 \node[shift=(a.corner 1), anchor=south, gray] {$1$};
	 \node[shift=(a.corner 2),anchor=south, gray] {$9$};
	 \node[shift=(a.corner 3),anchor=east, gray] {$8$};
	 \node[shift=(a.corner 4),anchor=east, gray] {$7$};
	 \node[shift=(a.corner 5),anchor=north, gray] {$6$};
	 \node[shift=(a.corner 6),anchor=north, gray] {$5$};
	 \node[shift=(a.corner 7),anchor=west, gray] {$4$};
	 \node[shift=(a.corner 8),anchor=west, gray] {$3$};
	 \node[shift=(a.corner 9),anchor=west, gray] {$2$};
	 \draw[fill=black] (-1.2,.3) circle (2pt);
	 \draw[fill=black] (.5,-.12) circle (2pt);
	 \draw[fill=black] (.5,-.82) circle (2pt);
	\draw[black, thick, ->] (-1.2,.3)--(t.side 1);
	\draw[black, thick, ->] (-1.2,.3)--(t.side 3);
	\draw[black, thick, ->] (a.side 2)--(-1.2,.4);
	\draw (0,.4) node  {$\blacksquare$};
	\draw[black, thick, ->] (a.side 8)--(.6,-.02);
	\draw[black, thick, ->] (.45, -.02)--(0.1,.4);
	\draw[black, thick, ->] (.55, -.12)--(t.side 7);
	\draw[black, thick, ->] (t.side 7)--(.6,-.72);
	\draw[black, thick, ->] (.56,-.72)--(.5,-.25);
	\draw[black, thick, ->] (.56,-.82)--(t.side 6);
	\node[shift=(a.side 1), black]{$\blacksquare$};
	\node[shift=(a.side 2), black]{$\blacksquare$};
	\node[shift=(a.side 3), black]{$\blacksquare$};
	\node[shift=(a.side 6), black]{$\blacksquare$};
	\node[shift=(a.side 7), black]{$\blacksquare$};
	\node[shift=(a.side 8), black]{$\blacksquare$};
	\draw (-.9,.3) node  {\tiny $x_{79}$};
	\draw (-.2,.6) node  {\tiny $x_{27}$};
	\draw (.25,-.12) node  {\tiny $x_{37}$};
	\draw (.25,-.82) node  {\tiny $x_{47}$};
	 \node[shift=(a.side 1), anchor=south, black]{\tiny $x_{19}$};
	 \node[shift=(a.side 2), anchor=east, black]{\tiny $x_{89}$};
	 \node[shift=(a.side 3), anchor=east, black]{\tiny $x_{78}$};
	 \node[shift=(a.side 6), anchor=west, black]{\tiny $x_{45}$};
	 \node[shift=(a.side 7), anchor=west, black]{\tiny $x_{34}$};
	 \node[shift=(a.side 8), anchor=west, black]{\tiny $x_{23}$};
\end{tikzpicture}
	}
  \caption{\label{fig:seed} In gray, a bicolored triangulation $\T$. In black, the seed $\Sigma_\T$. The distinguished boundary arcs are 
	$1 \to 7$ and $5 \to 7$.}
\end{figure}


\begin{defn}[Cluster variables]
        Let $a \to b$ with $a<b$ be an arc which is contained in a black polygon $P_i$ and is not the distinguished boundary arc $h_i \to j_i$. We define
	\[x_{ab}:=\frac{(-1)^{\area(a \to b)}\lrangle{Y Z_a Z_b}}{(-1)^{\area(h_i\to j_i)}\lrangle{Y Z_{h_i} Z_{j_i}}},\]
        which is a rational function on $\Gr_{k, k+2}(\CC)$.
\end{defn}

\begin{defn}[Seeds]
Let $\T \in \mathcal{S}(\overline{\mathcal{T}})$.
	We define the {quiver} $Q_{\T}$  as follows:
        \begin{itemize}
		\item Place a frozen vertex on each non-distinguished boundary arc of $P_1, \dots, P_r$ and a mutable vertex on every other arc 
			(a ``black arc'') bounding 
			a triangle of $\T$. 
                \item If arcs $a \to b$, $b \to c$, $c \to a$ form a triangle, we put arrows in $Q$ between the corresponding vertices, going clockwise around the triangle.
        \end{itemize}

        We label the vertex of $Q_\T$ on arc $a \to b$ of $\T$ with the function $x_{ab}$. The collection of vertex labels is the \emph{(extended) cluster} $\textbf{x}_{\T}$. The pair
        $(Q_{\T},\textbf{x}_\T)$
        is the \emph{seed} $\Sigma_\T$.
\end{defn}

See \cref{fig:seed} for an example.
Note that 
the cluster $\mathbf{x}_\T$ has size $2k$.


\begin{thm}
	[{\cite[Section 6.2]{PSBW}}]
	\label{thm:clustvar}
Let $\T \in \mathcal{S}(\overline{\mathcal{T}})$.
        Then 
        \[\mcv_\T:= \left \{Y \in \Gr_{k, k+2}(\CC):  \prod_{a \to b \text{ black arc of }\T}\lrangle{Y Z_a Z_b} \neq 0 \right\}\]
         is birational to an algebraic torus of dimension $2k$, and 
	 its 
        field of rational functions is 
        the field 
	$\CC(\mathbf{x}_\T)$
	of rational functions in the cluster $\mathbf{x}_\T$.

	The set
$		\mcv_{\overline{\T}}:= 
	\bigcup_{\T \in \mathcal{S}(\overline{\mathcal{T}})}
	\mcv_\T$
        is a cluster variety in $\Gr_{k, k+2}(\CC)$.
	In particular, 
if 
	$\T$ and $\T'$ are related by flipping  arc $a \to b$, 
	 seeds $\Sigma_\T$ and $\Sigma_{\T'}$ are related by mutation at $x_{ab}$.  
	\[\text{ The positive part }\mcv_{\overline{\T}}^{>0}:=\{Y \in \mcv_{\overline{\T}}: x_{ab}(Y)>0 \text{ for all cluster variables } x_{ab}\}\]
	of the cluster variety $\mcv_{\overline{\T}}$ 
        is equal to the positroid tile $\gto{\hatG(\overline{\T})}$.
\end{thm}

We conjecture that for even $m$, each positroid tile
of $\mathcal{A}_{n,k,m}(Z)$ can be realized as the totally positive part
of a cluster variety in $\Gr_{k,k+m}(\CC)$.

\section{Future directions}
Clearly many problems about the amplituhedron remain wide open.
Moreover, there are other 
related geometric objects, including 
the \emph{loop amplituhedron}
 \cite{arkani-hamed_trnka}, and the \emph{momentum amplituhedron}
(defined for $m=4$ in  \cite{mamp} and for even $m$ in \cite{LPW}).  It would be interesting to explore
these objects systematically as above.

\bigskip
\textsc{Acknowledgements:}
Main references for this article include \cite{karpwilliams,
karp:2017ouj, LPW, PSBW, SW}.
It is a pleasure to thank my  collaborators, 
especially S. Karp, T. Lukowski, M. Parisi,  M. Sherman-Bennett, D. Speyer, and Y. Zhang.
I have also benefited enormously from discussions
and collaborations with many others,  including 
F. Ardila, 
N. Arkani-Hamed, S. Fomin, A. Postnikov, K.
Rietsch,  F. Rinc\'on, and B. Sturmfels.
	Finally, I thank
	J. Boretsky, 
A. Burcroff, 
	C. Defant, S. Fomin, J. Gao, 
	S. Karp, B. Keller, M. Parisi, M. Sherman-Bennett, and R.
Vlad for helpful comments on this manuscript, and the 
	Radcliffe Institute for (virtual) hospitality.

This work was partially supported by NSF grants DMS-1854316 and DMS-1854512.

\bibliographystyle{emss}
\bibliography{bibliography}







\end{document}